\documentclass[aos]{imsart}
\usepackage[english]{babel}
\usepackage[utf8]{inputenc}
\usepackage{amsmath,amsthm, amssymb, latexsym, color,mathtools}
\usepackage{enumitem}
\usepackage{verbatim}
\usepackage{comment}
\usepackage{bbm}
\usepackage{mathrsfs}
\usepackage{tikz}
\usetikzlibrary{cd}
\usetikzlibrary{shapes,arrows,cd}
\usepackage{tikz-cd}
\usepackage[all,2cell]{xy}
\usepackage{abstract}
\usepackage{hyperref}
\usepackage[disable]{todonotes} 

\RequirePackage[colorlinks,citecolor=blue,urlcolor=blue]{hyperref}

\definecolor{org}{rgb}{1,0.53,0.0}
\newcommand{\ymm}[2][]{\todo[color=org!40,#1]{#2}}
\definecolor{ssw}{rgb}{0.1,0.45,0.1}
\newcommand{\ssw}[2][]{\todo[color=ssw!40,#1]{#2}}
\newtheorem{thm}{Theorem}
\newtheorem{defin}[thm]{Definition}     
\newtheorem{lem}[thm]{Lemma}
\newtheorem{cor}[thm]{Corollary}

\newtheorem{rem}[thm]{Remark}
\newtheorem{prop}[thm]{Proposition}

\newtheorem{ass}[thm]{Assumption}
\numberwithin{thm}{section}
\numberwithin{equation}{section}

\newcommand{\R}{\mathbb{R}}

\newcommand{\Z}{\mathbb{Z}}
\newcommand{\N}{\mathbb{N}}

\newcommand{\KR}{\text{KR}}

\newcommand{\pKR}{\hat{p}_{\text{KR}}}

\newcommand{\eps}{\varepsilon}

\newcommand{\sven}[1]{{\textcolor{red}{#1}\index{#1}}}

\startlocaldefs
\endlocaldefs

\begin{document}
\begin{frontmatter}
	\title{On minimax density estimation via measure transport}
\begin{aug}

\author[]{\fnms{Sven}~\snm{Wang}\ead[label=e1]{svenwang@mit.edu}}
\and
\author[]{\fnms{Youssef}~\snm{Marzouk}\ead[label=e2]{ymarz@mit.edu}}

\address[]{Massachusetts Institute of Technology\\ Cambridge, MA 02139 USA\\ \printead[]{e1,e2}}
\end{aug}

\end{frontmatter}

\begin{abstract}
We study the convergence properties, in Hellinger and related distances, of nonparametric density estimators based on measure transport. These estimators represent the measure of interest as the pushforward of a chosen reference distribution under a
transport map, where
%
%
the map is chosen via a maximum likelihood objective (equivalently, an empirical Kullback--Leibler loss) or a penalized version thereof.
We establish concentration inequalities for a general class of penalized measure transport estimators, 
by combining techniques from M-estimation with analytical properties of the transport-based density representation.
We then demonstrate the implications of our theory for the case of triangular Knothe--Rosenblatt (KR) transports on the $d$-dimensional unit cube, and show that both penalized and unpenalized versions of such estimators achieve minimax optimal convergence rates over H\"older classes of densities.
Specifically, we establish optimal rates for unpenalized nonparametric maximum likelihood estimation over bounded H\"older-type balls, and then for certain Sobolev-penalized estimators and sieved wavelet estimators.

\end{abstract}

\tableofcontents
\section{Introduction}
\label{sec:intro}
A natural way to represent complicated probability measures is through transportation. The core idea is to express the measure of interest $P_0$ as the pushforward of a fixed reference measure $\eta$ through some map $T$, i.e., $P_0 = T_\# \eta$. If the goal is to learn $P_0$ through data, the problem thus becomes one of learning $T$. With an estimate $\hat T$ of $T$ in hand, it is easy to generate approximate samples from $P_0$, so long as one can sample from $\eta$ and evaluate $\hat T$ at a low cost. Similarly, assuming that $P_0$ and $\eta$ have densities with respect to a common dominating measure and that $\hat T$ is invertible (with a continuously differentiable inverse), then the change-of-variables formula immediately yields an estimate ${\hat T}_\# \eta (x)= \eta(\hat T^{-1}(x))\det \nabla \hat T^{-1} (x)$ for the density of $P_0$. 

A notable feature of transport-based representations is that they naturally incorporate the constraints imposed by the simplex of probability measures (i.e., non-negativity and integration to one). This feature is in contrast with common (both linear and nonlinear) approximation schemes employing, e.g., kernel projections \cite{GL14, GN16} or wavelets \cite{DJKP96}, where such constraints are at most approximately satisfied or require artificial enforcement. 
%
%

The transport construction has been used successfully in (i) variational inference \cite{el2012bayesian,RM15,durkan2019neural}, where the density of $P_0$ is accessible up to a normalizing constant; (ii) density estimation \cite{tabak2010density,tabak2013family,anderes2012general,dinh2017,PPM17,jaini2019}, where one is given a sample from $P_0$ and wishes to characterize its density; and (iii) generative modeling \cite{kingma2018}, where one aims to simulate additional samples from $P_0$ given an initial training set. Empirically, transport-based methodologies achieve state-of-the-art performance in many of these tasks \cite{de2020block,grathwohl2018ffjord,durkan2019neural,finlay2020train}, with additional uses ranging from likelihood-free inference \cite{papamakarios2019sequential,lueckmann2019likelihood,greenberg2019automatic} to anomaly detection \cite{gudovskiy2022cflow}.
%
%
Throughout these applications, a variety of parameterizations have been proposed for $T$, using both classical function approximation schemes (such as polynomials or wavelets) and neural networks. It is useful to recognize that there are, in general, infinitely many transport maps $T$ that can deterministically couple two given (absolutely continuous) probability measures. Normalizing flows, popular in the machine learning community \cite{papamakarios2021normalizing,kobyzev2020normalizing}, employ specific parameterizations of transport maps that ensure invertibility and easy evaluation of the Jacobian determinant $\det \nabla T^{-1}$ by construction. Neural ordinary differential equations (ODEs) \cite{haber2017stable,chen2018neural,finlay2020train} provide an alternative way of guaranteeing invertibility, by representing $T$ as the finite-time flow map of a deterministic ODE system. Optimal transport maps provide a canonical choice in certain applications, e.g. \cite{burger2012regularized}, especially when the geometric structure induced by the Wasserstein distance plays a role. 

%

Another natural class of transport schemes are triangular maps \cite{bogachev2005triangular}. They play a central role in many of the aforementioned constructions, e.g., as the building block of ``autoregressive'' normalizing flows, a format that includes many of the most prevalent current approaches. Such normalizing flows represent maps $T$ 
as a composition of several triangular maps, interleaved with coordinate permutations. 
Since triangular maps admit an explicit construction in terms of ordered marginal conditional distribution functions (see Section \ref{MLE-res} for details), they are also especially suited for conditional density estimation \cite{BMZ20} and conditional simulation \cite{MMPS16}. Moreover, they inherit sparsity from the conditional independence structure of $P_0$ and $\eta$, and thus have important connections to probabilistic graphical models; see \cite{spantini2018inference} for a full discussion. Other applications of triangular maps include simulation-based inference \cite{baptista2022bayesian} and nonlinear filtering \cite{spantini2019coupling}. Both the evaluation of Jacobian determinants and numerical inversion are convenient for triangular maps, making them an appealing choice from a computational perspective. The concrete examples studied below---i.e., to which the general theory developed in Section \ref{sec-penalty} is applied---will arise from triangular maps.

\smallskip 

Despite their empirical success, we still have little understanding about the theoretical properties of transport-based methods. (The few exceptions, notably the paper \cite{ISPH21} which is related to ours, will be discussed further below.) Key questions include:
Under what conditions are transport-based density estimators consistent, and what are statistical large-sample convergence rates for transport-based methods? How do those rates compare to statistically optimal minimax convergence rates (which are known to be achieved by classical statistical methods for learning probability distributions, e.g., based on wavelets \cite{DJKP96} or kernel density estimation \cite{GL11})? What are natural parameterization and penalization schemes for transport maps?

The present paper aims to shed light on these theoretical questions. Our main focus is the well-known setting of nonparametric density estimation; see, e.g., \cite{T08}. Here i.i.d.\ observations $(X_i:1\le i\le N)$ are given from some unknown `ground truth' probability measure $P_0$ with Lebesgue density $p_0$ on some bounded, open domain $U \subseteq \mathbb{R}^d$, and the goal is to infer $p_0$. The class of estimators we consider arises from a \textit{penalized maximum likelihood} objective, given by
\begin{equation}\label{flashy-objective}
\begin{split}
    \hat S\in & \arg\min_{S\in \mathcal S}  \Big\{ -\sum_{i=1}^N \log S^\# \eta (X_i)+\textnormal{pen}(S) \Big\}.
\end{split}
\end{equation}
Here, $\eta$ is a fixed reference density, $\mathcal S$ is a class of \textit{invertible} $C^1$-diffeomorphic transport maps, $S^\#(\cdot)$ denotes the pullback operation, and $\text{pen}:\mathcal S\to [0,\infty)$ is some \textit{penalty functional}; see Section \ref{sec-penalty} below for details. The resulting estimator for $p_0$ is then given as the plug-in $\hat S^\# \eta \approx p_0$. The objective (\ref{flashy-objective}) can equivalently be reformulated as an optimization problem over maps $T=S^{-1}$, replacing the pullback by a pushforward. However, (\ref{flashy-objective}) is used more frequently for computational reasons; we thus follow this convention. Recall also that (\ref{flashy-objective}) corresponds to minimizing an empirical approximation of the Kullback--Leibler divergence, which often appears as a loss function; if $X_1,\dots, X_N \stackrel{\text{i.i.d.}}{\sim} P_0$, then $\textnormal{KL}(p_0 \| S^\# \eta) \approx -\frac 1N\sum_{i=1}^N \log S^\# \eta (X_i) +\int p_0\log p_0$.

\smallskip

\paragraph{Contributions}
Our main contributions are twofold. First, we prove a concentration bound for general transport-based density estimators arising from variational problems of the form \eqref{flashy-objective}; see Theorem \ref{thm-penalty-gen} below. Informally speaking, this result characterizes the convergence rate of $\hat S^\#\eta$ towards $p_0$ in the Hellinger distance based on two fundamental quantities of the pair $(\mathcal S, \textnormal{pen}(\cdot) )$: (i) a metric entropy bound for the model class $\mathcal S$ in some suitable H\"older-type norm, and (ii) the `best approximation' of the data-generating density $p_0$ by an some element $S^*\in\mathcal S$. Our concentration inequality is applicable in principle to any choice of transport class and penalty, and provides an explicit framework for understanding the impact of different choices of $\mathcal S$ and $\textnormal{pen}(\cdot)$. To our knowledge, penalized estimators of transport maps have not been studied in previous work. In proving this result, we adapt classical methods for the study of nonparametric M-estimators based on empirical process theory, developed in \cite{VDG00,VDG01,WS95}, to transport-based estimators of the form (\ref{flashy-objective}), by combining them with analytical properties of the transport-based representation $S\mapsto S^\# \eta=(\eta \circ S) \det \nabla S$ of the probability density of interest; see Section \ref{gen-pfs} and in particular the proof of Theorem \ref{thm-penalty-gen} for details.


Our second main contribution is that we give a comprehensive study of the implications of the general theory for the case of triangular `Knothe--Rosenblatt' (KR) maps of the form
\begin{equation}\label{KR-intro}
    S(x)=\begin{bmatrix*}[l]
	S_1(x_1) \\
	S_2(x_1, x_2) \\
	~~\vdots \\
	S_d(x_1,\dots,x_d)
\end{bmatrix*},
\end{equation}
between $d$-dimensional bounded product domains. In total, three types of concrete estimators are studied: (i) unpenalized nonparametric MLE over bounded H\"older-type balls, (ii) penalized estimators with Tikhonov-type Sobolev norm penalty, and (iii) sieved penalized estimators based on truncated wavelet expansions. While the first estimator might appear simplest on the conceptual level, it may be infeasible to realize computationally; in contrast, the latter two penalization-based methods permit the use of standard optimization methods for unconstrained problems on 
Euclidean spaces. All three estimators are proven to achieve statistically minimax-optimal convergence rates in the Hellinger (and then also $L^2$) distance, demonstrating that transport methods can satisfy the same theoretical guarantees as classical wavelet estimators or kernel density estimators (see, e.g., \cite{DJKP96, GN16, GL11} and references therein). To obtain those minimax-optimal results, we require a detailed characterization of the smoothness properties of KR transport maps between H\"older-regular densities in terms of certain natural \textit{anisotropic} norms (partly noted also in the very recent work \cite{ISPH21}). We then propose a parameterization of triangular maps which naturally captures these anisotropic smoothness properties, and which is used crucially in the construction of penalized estimators (ii) and (iii); see Sections \ref{MLE-res} and \ref{triangle-II} below for details.


\paragraph{Further related work}

Key methodological efforts were reviewed above, so we focus our discussion here on recent theoretical advances related to this paper. The papers \cite{ZM20,zech2022sparse} study the approximation of triangular transport maps between analytic densities, by rational functions and neural networks, on high- and infinite-dimensional bounded domains. \cite{BMZ20} proposes a relatively general parameterization of monotone triangular maps on unbounded domains, and analyzes properties of the optimization problem (\ref{flashy-objective}) (without penalization) under such parameterizations. None of these three papers \cite{ZM20,zech2022sparse,BMZ20} addresses questions of statistical consistency, however. 

The paper \cite{ISPH21} appeared while we were preparing this manuscript. It derives convergence rates for non-penalized maximum likelihood estimators of triangular maps over H\"older-type balls. The setting considered there is similar to ours, but our proof techniques differ substantially: the approach in \cite{ISPH21} is to directly derive a concentration inequality for the Kullback--Leibler objective, while we exploit the connection to maximum likelihood estimation. In fact, our results in Section \ref{MLE-res} improve upon the rates derived in \cite{ISPH21}; see Corollary \ref{cor-MLE} below and the ensuing discussion.




Another active line of work surrounds the statistical estimation of Brenier (optimal transport) maps between two distributions, given independent samples from each. The paper \cite{HR20} establishes minimax upper and lower bounds, under smoothness assumptions on the optimal transport map. 
\cite{MBNW21} build on this work and discuss estimators of the optimal transport map under a variety of smoothness assumptions. Interestingly, one proposed minimax-optimal estimator of the optimal transport map relies on first estimating the target and reference densities using nonparametric kernel or wavelet techniques.
Thus, while related, the setting of these papers differs from ours, since we are interested in estimation of the \textit{target density}, and only secondarily in the map itself (cf.\ Corollary \ref{cor-map} below); the role of the transport map and the reference distribution is to provide a representation of the target. When the reference is known, results in \cite{HR20} (see Remark 5 in that paper) yield estimation rates also for the target in Wasserstein distance. \cite{NB22} also consider the estimation of smooth densities in Wasserstein distance, and propose several constructions that achieve minimax rates in this setting.

\paragraph{Structure of the paper}

Section \ref{MLE-res} first illustrates, in a concrete setting, some of the main intuitions of our approach---which underlie the more abstract conditions laid out in Section \ref{sec-penalty}. In particular, Section \ref{MLE-res} recalls the construction of the Knothe--Rosenblatt rearrangement and derives its basic anisotropic regularity properties, which in turn yield the initial Theorem \ref{thm-anisotropic} demonstrating that unpenalized transport-based density estimators are able to achieve minimax convergence rates. Section \ref{sec-penalty} introduces our main result, specifically Theorem \ref{thm-penalty-gen} establishing concentration bounds for general penalized transport-based density estimators. Section \ref{triangle-II} then studies parameterizations of triangular maps and the implications of our general results for two representative penalized estimation procedures. The proofs are given in Sections \ref{gen-pfs}--\ref{MLE-proofs}.




\section{Triangular maps I: maximum likelihood}\label{MLE-res}
\subsubsection{Construction of Knothe--Rosenblatt maps}
%
To begin, we recall the standard construction of the Knothe--Rosenblatt rearrangement, cf.\ \cite{S15,ZM20}. To limit technical intricacies regarding the regularity properties of such transport maps, in this section we consider the case of product domains. Specifically, let $Q_d=[0,1]^d$ denote the $d$-dimensional unit cube, and suppose that $\nu,\mu$ are two strictly positive densities on $Q_d$. Then there exists a unique triangular map $S_{\nu,\mu}$ satisfying $(S_{\nu,\mu})_\#\nu=\mu$ (where $_\#$ denotes the pushforward operation, which we apply to densities in a slight abuse of notation) such that for all $k=1,\dots,d$, the component function $(S_{\nu,\mu})_k: Q_k \to [0,1]$ is monotone increasing in the $k$-th variable. The map is constructed as follows. Writing $x_{1:k}=(x_1,\dots, x_k)\in Q_k$, we define the marginal densities
	\begin{equation}\label{eta-def}
		\begin{split}
			\tilde \nu_k (x_{1:k})&= \int_{Q_{d-k}} \nu(x_{1:k},y_{1:d-k}) dy_{1:d-k}, ~~~1\le k \le d,\\
		\end{split}
	\end{equation}
while for $k=0$ we set $\tilde \nu_0\equiv 1$. Then, define the conditional densities of the $k$-th variable given the first $k-1$ variables by
\begin{equation}\label{nu-k-def}
\nu_k(x_{1:k})= \frac{\tilde \nu_k(x_{1:k})}{\tilde \nu_{k-1}(x_{1:k-1})}, ~~~ 1\le k \le d.
\end{equation}
Integrating $\nu_k (x_{1:k})$ along $x_k$ gives rise to the corresponding conditional cumulative distribution functions (CDFs)
	\begin{equation}\label{Fdef}
		F_{k,x_{1:k-1}}^\nu:[0,1]\to [0,1],~~~ F_{k,x_{1:k-1}}^\nu(x_k)\coloneqq \int_{0}^{x_k}\nu_k(x_{1:k-1},y)dy. 
	\end{equation}
For $\mu$, we analogously define the marginal conditional densities $\mu_k(x_{1:k})$, $1\le k\le d$, with CDFs $F^\mu_{k,x_{1:k-1}}$.

For notational convenience, let us write $S=S_{\nu,\mu}$. Then, the first component $S_1$ of $S$ is given by
\begin{equation*}
	\begin{split}
		S_1(x_1)\coloneqq \big[(F_{1}^\mu)^{-1}\circ F_{1}^\nu \big](x_1).
	\end{split}
\end{equation*}
The remaining components $S_2,\dots ,S_d$ are defined recursively: given $S_1,\dots,S_{k-1}$ and writing
\[S_{1:k-1}(x_{1:k-1})\equiv \begin{bmatrix*}[l]S_1(x_1)\\ ~~ \vdots\\ S_{k-1}(x_1,\dots,x_{k-1})\end{bmatrix*}:Q_{k-1}\to Q_{k-1},\]
we define
\begin{equation}\label{Tk-def}
	\begin{split}
		S_k(x_{1:k})\coloneqq \big[\big(F^\mu_{k,S_{1:k-1}(x_{1:k-1})}\big)^{-1}\circ F_{k,x_{1:k-1}}^{\nu} \big] (x_k).
	\end{split}
\end{equation}
Monotonicity of each component $S_k$---specifically, the property that $\partial_{x_k} S_k > 0$ for all $x_{1:k} \in Q_k$---follows from the fact that $F_{k,x_{1:k-1}}^{\nu}$ and $F^\mu_{k,S_{1:k-1}(x_{1:k-1})}$ are strictly increasing functions for strictly positive densities $\nu, \mu$. We call $S=S_{\nu,\mu}$ the \textit{Knothe--Rosenblatt (KR) rearrangement} 
from $\nu$ to $\mu$. 

It is useful to note that for $d=1$, the KR rearrangement constructed above coincides with the optimal transport map for any strictly convex transport cost \cite[Thm.\ 2.9]{S15}; for the atomless and fully supported measures considered here, both are explicitly given by the composition $(F^\mu)^{-1}\circ F^\nu$ of the CDFs. As is apparent from \eqref{Tk-def}, the KR map essentially generalizes this increasing rearrangement to $d \geq 2$ dimensions, by recursively applying it to a sequence of marginal conditional distributions. \ssw{Another link between the KR map and optimal transport is described by \cite{carlier2010}, who show that the KR map is the limit of a sequence of optimal transport maps obtained with increasingly anisotropic quadratic costs.} 



\subsubsection{Anisotropic regularity of KR maps} 

The explicit construction of KR maps allows us to study regularity estimates for $S_{\nu,\mu}$ directly from the definitions (while deriving regularity properties for optimal transport maps between probability measures is more intricate, see, e.g., \cite{V09,F17}). It is well known that a KR map $S_{\nu,\mu}$ possesses the same regularity as that of the densities $\nu,\mu$ it couples \cite{S15}. However, we require a slightly refined result, Proposition \ref{prop-fact}, which characterizes an \textit{anisotropic} class of transport maps corresponding to `$\alpha$-smooth' reference and target distributions. From a statistical point of view, this is key to inform an effective choice of $\mathcal S$ in \eqref{flashy-objective} over which to optimize.

\smallskip 

For any open subset $U\subseteq \R^d$ and integer $\alpha>0$, let $C^\alpha(U)$ denote the space of real-valued $\alpha$-times differentiable functions on $U$ with uniformly continuous derivatives, normed by
\[ \|f\|_{C^\alpha(U)}\coloneqq\sum_{0\le|\beta|\le \alpha}\|\partial_\beta f\|_\infty.\]
For non-integer $\alpha>0$, we denote by $C^\alpha(U)$ the space of $\alpha$-H\"older continuous functions with norm
\[ \|f\|_{C^\alpha(U)}\coloneqq\sum_{0\le|\beta|\le \alpha}\|\partial_\beta f\|_\infty+\sum_{|\beta| = \lfloor \alpha\rfloor}\sup_{x\neq y}\frac{|\partial_\beta f(x)-\partial_\beta f(y)|}{|x-y|^{\beta-\lfloor \beta\rfloor}}.\]
Moreover, for $k\ge 1$, we denote the corresponding vector-valued H\"older spaces mapping into $\R^k$ and $Q_k$ by $C^\alpha(U, \R^k)$ and $C^\alpha(U, Q_k)$, respectively. (The notation $C^\alpha(U)$ always refers to the real-valued case $k=1$.)
We write $f\in C^\alpha(\bar U)$ for any $f\in C^\alpha(U)$ which (along with all its derivatives) possess continuous extensions to $\bar U$.

\smallskip

Now, for any $0<c<B<\infty$ and $\alpha >0$, let us define the regularity class of densities
\begin{equation}\label{Mdef}
	\mathcal M(\alpha,B,c)= \Big\{ \nu\in C^\alpha(Q_d):~ \nu\ge c, ~\|\nu\|_{C^\alpha} \le B,~\int_{Q_d} \nu =1 \Big\},
\end{equation}
along with the subclass of densities which factorize into $\alpha$-smooth marginals,
\begin{equation}\label{M-fact}
	\tilde{\mathcal M}(\alpha,B,c)= \Big\{ \nu\in \mathcal M(\alpha,B,c):~\nu(x)= \Pi_{k=1}^d v_k(x_k)~\textnormal{for some}~ \|v_k\|_{C^\alpha([0,1],\R)}\le B \Big\}.
\end{equation}
While $\mathcal M(\alpha,B,c)$ is a natural class in which to model the `ground truth' density (which here is taken to be $\nu$), to limit technicalities we will assume the reference density $\mu$---which we may choose arbitrarily---to belong to $\tilde{\mathcal M}(\alpha,B,c)$.

As can be seen from the construction preceding (\ref{Tk-def}), intuitively the $k$-th component of a KR map $S_{\nu,\mu}$ should possess higher regularity in the $k$-th variable than in the previous variables, owing to the integration in (\ref{Fdef}). Formally, for $L>0, c_{\text{mon}}>0$ and integer $\alpha \ge 1$, let us define the regularity class
\begin{equation}\label{KR}
	\begin{split}
		\text{KR}(\alpha, L, c_{\text{mon}})\coloneqq \Big\{ S\in C^\alpha(Q_d,Q_d)&~\textnormal{bijective and triangular s.t.}~\forall \, 1\le k\le d:\\
	    &\max\big( \|S_k\|_{C^\alpha(Q_k)},\|\partial_kS_k\big\|_{C^\alpha(Q_k)}\big)\le L,\\
	    &\inf_{x_{1:k}\in Q_k} \partial_kS_k(x_{1:k})>c_{\text{mon}} \Big\}.
	\end{split}
\end{equation}
Here, $c_{\text{mon}}>0$ is a lower bound for the least amount by which each component map $S_k$ must be increasing in the $k$-th variable.

\begin{prop}\label{prop-fact}
   For any $\alpha\ge 1$ and constants $0<c<B<\infty$, there exist $L,c_{\text{mon}}>0$ (depending only on $c,B,d,\mathcal O$) such that for any $\nu\in \mathcal M(\alpha,B,c)$ and $\mu\in \tilde{\mathcal M}(\alpha,B,c)$, we have $S_{\nu,\mu}\in \KR(\alpha,L,c_{\text{mon}})$.
\end{prop}

The anisotropic regularity in $\KR(\alpha,L,c_{\text{mon}})$ is `sharp' insofar as one may also show the converse of the above proposition: namely that pullbacks of $\alpha$-smooth densities under transports from $\KR(\alpha,L,c_{\text{mon}})$ are again $\alpha$-smooth.
\begin{prop}\label{prop-sharp} For any $L,c_{\text{mon}}>0$, and $0<c<B<\infty$, there exist constants $\tilde c, \tilde B$ such that for every $\mu \in \tilde{\mathcal M}(\alpha,B,c)$ and every $S\in \KR(\alpha,L,c_{\text{mon}})$, \[S^\#\mu \in \mathcal M(\alpha,\tilde B,\tilde c).\]
\end{prop}
The proofs of the preceding results follow straightforwardly from the definitions, and can be found in Section \ref{MLE-proofs} below. Note that requiring $S$ to merely belong to $C^\alpha$ in Proposition \ref{prop-sharp} may \textit{not} be enough to guarantee that $S^\#\mu \in C^\alpha$. 
In this sense, while elementary, the above propositions provide crucial insight to inform which classes of triangular transports are able to achieve statistically optimal convergence rates.

\subsubsection{Statistical convergence rates} 

We now present a first theorem which demonstrates that maximum likelihood estimators based on the anisotropic classes of triangular maps described above achieve minimax rates of convergence for H\"older-smooth densities. Consider the classical nonparametric density estimation setup, in which we observe independent, identically distributed (i.i.d.) data
\[ X_1,\dots, X_N \in Q_d,~~~ X_i \stackrel{\text{i.i.d.}}{\sim} P_0, \]
from some unknown probability measure $P_0$ on $Q_d$ (whose $n$-fold product measure we denote by $P_0^N$). Define the following negative log-likelihood objective over a class of transport maps $\mathcal S$,
\begin{equation}\label{objective}
    \mathcal J_N(S)\coloneqq-\sum_{i=1}^N\log S^\#\eta (X_i)=-\sum_{i=1}^N\log [\eta(S(X_i))\det \nabla S(X_i)],~~~ S\in \mathcal S.
\end{equation}
Using Propositions \ref{prop-fact} and \ref{prop-sharp}, we obtain the following concentration inequality for density estimators $\hat S^\# \eta$ arising from $(\ref{objective})$. We denote the usual Hellinger distance between any two densities $p,\tilde p\in L^1(Q_d)$ by
\[ h(p,\tilde p) = \Big(\int_{Q_d} \big( \sqrt p- \sqrt{\tilde p} \big)^2 dx\Big)^{1/2}. \]

\begin{thm}\label{thm-anisotropic}
    Let $0<c<B<\infty$ and $\alpha>d/2$, and fix any reference density $\eta \in \tilde{\mathcal M}(\alpha,B,c)$. Then there are constants $L,c_{\text{mon}}>0$ and $C>0$ such that for any $p_0\in \mathcal M(\alpha,B,c)$, any $N\ge 1$ and any minimizer
    \[ \hat S\in \arg\min_{S\in \mathcal S}\mathcal J_N(S) ~~\textnormal{with}~~\mathcal S= \KR(\alpha,L,c_{\text{mon}}), \] 
    we have the concentration inequality
    \[ P_0^N\big( h(\hat S^\#\eta,p_0)\ge CN^{-\frac{\alpha}{2\alpha+d}} \big)\le C\exp\Big( -\frac {N^{d/(2\alpha+d)}}{C}  \Big). \]
\end{thm}
The assumption for $p_0$ to be lower bounded across $Q_d$ is made here to ensure regularity properties of the KR-rearrangement $S_{p_0,\eta}$, while the assumption $\alpha>d/2$ is needed for standard entropy integrals of H\"older-type spaces \cite{NP07} to converge. Relaxing those assumptions may be possible at the expense of further technical intricacies, but we do not pursue this here.

Since $p_0\ge c$ and $\sqrt x \lesssim x$ on $[c,\infty)$ we immediately obtain the same rate of convergence for the $L^2$-distance,
\begin{equation}\label{L2-rate}
    \|\hat S^\#\eta-p_0\|_{L^2(Q_d)} = O_{P_0^N}(N^{-\frac\alpha{2\alpha+d}}),
\end{equation}
as well as the total variation (i.e., $L^1$) distance. The preceding convergence rates correspond to the optimal minimax $L^p$-risk of estimation of $\alpha$-smooth densities on bounded domains \cite{DJKP96, DJ98, GL14}.

Our results above also imply the following bound for the Kullback--Leibler risk. It follows essentially from the fact that for classes of upper and lower bounded densities, the KL divergence behaves like the \textit{squared} Hellinger distance.
\begin{cor}\label{cor-MLE}
In the setting of Theorem  \ref{thm-anisotropic}, we have that
\begin{equation}\label{KL-rate}
    \textnormal{KL}(p_0,\hat S^\# \eta) = O_{P_0^N}(N^{-\frac{2\alpha}{2\alpha+d}}).
\end{equation} 
\end{cor}
Since the proof of Corollary \ref{cor-MLE} follows the same lines as that of Corollary \ref{thm-KL}, it is omitted. Note that in the limit of the smoothness $\alpha$ tending to infinity, the rate exponent in (\ref{KL-rate}) satisfies $\frac{2\alpha}{2\alpha+d}\to 1$. The very recent preprint \cite{ISPH21} studies nonparametric transport-based MLEs over anisotropic classes which are very similar to (\ref{KR}), under Kullback--Leibler loss. In the regularity regime $\alpha >d/2$ considered here, their results yield the bound $\textnormal{KL}(p_0,\hat S^\# \eta) = O_{P_0^N} (N^{-1/2})$ (see Theorem 3.5 in \cite{ISPH21}), which is slower by a polynomial factor than our rate $N^{-\frac{2\alpha}{2\alpha+d}}$. The key difference between their proof strategy and ours (and a possible source of suboptimality) is that they study directly the concentration of the KL objective around its population version, rather than employing the connection to maximum likelihood estimation; we refer to \cite{ISPH21} and Section \ref{MLE-proofs} for details.

\section{General penalized estimators}\label{sec-penalty}

Our first Theorem \ref{thm-anisotropic} is specific to maximum-likelihood type estimators over the anisotropic class $\KR (\alpha,L,c_{\text{mon}})$ from (\ref{KR}). While this is of theoretical appeal, in practice it may be challenging to optimize over the  infinite-dimensional bounded convex set $\KR (\alpha,L,c_{\text{mon}})$, and it may be more favorable to employ various re-parameterizations and finite-dimensional discretizations of transport maps instead.

The goal of this section is to provide a convergence result, Theorem \ref{thm-penalty-gen}, for a general class of transport-based estimators. The estimators arise as minimizers of objective functions which may include a penalty term, where the penalty acts on some parameter $\theta\in \Theta$ indexing a set of transport maps. This formulation permits the analysis of estimators based on reformulations of constrained optimization problems (such as over $\KR (\alpha,L,c_{\text{mon}})$) as numerically more appealing unconstrained objectives. Concrete examples will be provided in Section \ref{triangle-II}.

Since our results in this section apply in principle also to classes of non-triangular transport maps, and in particular do not hinge on the specific regularity properties from Proposition \ref{prop-fact} valid on product domains, we consider here a setup where the respective supports of the data-generating density and reference $\eta$ are not necessarily product domains. For some (and not necessarily identical) bounded connected open sets $U,V\subseteq \R^d$, suppose that $P_0$ is a Borel probability measure on $U$ with Lebesgue density $p_0:U\to (0,\infty)$ generating data
\begin{equation}\label{data}
    X_1,\dots,X_N \stackrel{\text{i.i.d.}}{\sim} P_0,
\end{equation}
and that $\eta:V\to (0,\infty)>0$ is some reference density.

\smallskip

Now, suppose that $\Theta$ is any parameter set which indexes a set of diffeomorphisms $S_\theta: U\to V$, which we denote by
\[\mathcal S\coloneqq\{S_\theta:~\theta\in \Theta\}.\]
Then, for some penalty functional $\textnormal{pen}:\Theta \to \R \cup \{ +\infty \}$ and scalar `regularization parameter' $\lambda>0$, we consider the objective
\begin{equation}\label{JN}
    \mathcal J (\theta) \equiv \mathcal J_{N,\lambda}(\theta) = - \frac 1N\sum_{i=1}^N \log\big[ \eta(S_\theta(X_i)) \, \det \nabla S_\theta(X_i)\big] + \lambda^2\textnormal{pen}(\theta)^2.
\end{equation}
Note that the penalty in (\ref{JN}) acts on the parameter $\theta$ directly instead of the transport map $S_\theta$. We introduce this more general formulation to permit later also the study of penalized estimators with parameterizations $\theta\mapsto S_\theta$ which are not necessarily unique, but which may be attractive from a practical numerical standpoint, e.g., the rational parameterizations detailed in Section \ref{triangle-II} below, or neural network parameterizations \cite{PPM17,ZM20}.

As seen in the previous section, for triangular maps the `diagonal' derivatives play a distinguished role in characterizing regularity properties, and accounting for this fact is key for obtaining sharp convergence rates. To this end, for any multivariate function $S\in C^1(U,\R^d)$, we define the anisotropic norm
\[ \|S\|_{C^1_{\text{diag}}}\coloneqq\sum_{k=1}^d \|S_k\|_\infty + \sum_{k=1}^d \|\partial_kS_k \|_\infty. \]
The following assumption, which distinguishes between the triangular case and the general case, entails some important `global' boundedness conditions which we require to hold for the class of transport maps $\mathcal S$ induced by the parameter space $\Theta$.

\begin{ass}\label{generic-assumption}
Suppose that either of the following two statements holds.
\begin{enumerate}[label=\bfseries(\roman*)]
\item $\mathcal S\subseteq C^1(U,V)$ is a class of diffeomorphisms satisfying that for some $M>0$,
\begin{equation}\label{C1-bd}
    \sup_{S\in\mathcal S} \|S\|_{C^{1}}\le M~~~\textnormal{and}~~~     \inf_{S\in\mathcal S}\inf_{x\in U}\det \nabla S(x) \ge M^{-1}.
\end{equation}
    \item $\mathcal S\subseteq C^1(U,V)$ is a class of \emph{triangular} maps such that for some $M>0$,
\begin{align}
    \label{C1-diag-bd}
    \sup_{S\in\mathcal S} \|S\|_{C^{1}_{\text{diag}}}&\le M,\\
    \inf_{S\in\mathcal S}\inf_{x\in Q_d}\Pi_{k=1}^d\partial_kS_k(x_{1:k}) &\ge M^{-1}.\label{general-lb}
\end{align}
\end{enumerate}
\end{ass}

For any $\theta\in\Theta$ and Lebesgue probability density $p:U\to [0,\infty)$, define the functional
\begin{equation}\label{tau-def}
    \tau^2(\theta,p)\coloneqq h^2(S_\theta^\#\eta,p)+\lambda^2 \textnormal{pen}(\theta)^2.
\end{equation}
For any element $\theta_*\in \Theta$, define the classes
\begin{align}
    \Theta_*(\lambda,R)& \coloneqq \{\theta\in \Theta: \tau^2(\theta,S_{\theta_*}^\#\eta)\le R^2 \},\\
    ~~~~~\mathcal S_*(\lambda,R)& \coloneqq\{S_\theta: \theta\in \Theta_*(\lambda,R)\}. \label{theta-*}
\end{align}
In what follows, assume that a minimizer of (\ref{JN}) over $\mathcal S$ exists, and let $\hat \theta$ denote any minimizer (else if $\arg\min_{\theta\in \Theta}\mathcal J_{N,\lambda}= \emptyset$ the theorem holds true for approximate minimizers). For any metric space $(X,d)$, subset $A\subseteq X$, and $\rho>0$, we denote by $H(A,d,\rho)$ the usual metric entropy of the set $A$ (i.e., the logarithm of the size of a minimal $\rho$-covering of $A$).

\begin{thm}\label{thm-penalty-gen}
Suppose that the data are given as in (\ref{data}), and that Assumption \ref{generic-assumption} holds. In the case \textbf{(i)} define $|\cdot|\coloneqq \|\cdot\|_{C^1}$, whereas in the case \textbf{(ii)} define $|\cdot|\coloneqq \|\cdot\|_{C^1_{\text{diag}}}$. Assume moreover that for some constants $K,\tilde K>0$, we have that
\[ p_0\le K, ~~~ \tilde K^{-1}\le \eta \le \tilde K, \]
and that $\eta$ is Lipschitz continuous on $V$. For any $\theta_*\in \Theta$, define the entropy integral
\begin{equation}\label{entr-integral}
    \mathcal J_*(\lambda,R)\coloneqq R+ \int_0^{R} H^{1/2}(\mathcal S_*(\lambda,R),|\cdot|,\rho)d\rho,
\end{equation}
and let $\Psi(\lambda, R)\ge \mathcal J_*(\lambda,R)$ denote any upper bound such that for each $\lambda>0$, $R\mapsto \Psi(\lambda,R)/R^2$ is non-increasing.

\smallskip 

Then, there exist constants $c,c'>0$ only depending on $U,V,M,K,\tilde K$ and $\|\eta\|_{Lip}$ such that for all $\theta_*,\lambda,N,\delta$ satisfying
\begin{equation}\label{delta-req}
    \sqrt N\delta^2 \ge \Psi(\lambda,\delta),
\end{equation}
and for any $t\ge \delta$, any minimizer  $\hat \theta_N$ of $\mathcal J_{N,\lambda}$ from (\ref{JN}) satisfies
\begin{equation}\label{concentration}
    \begin{split}
        P_0^N\left (\tau^2(\hat\theta_N,p_0) \ge 2\big[\tau^2(\theta_*,p_0) +t^2\big] \right ) \le c\exp\left (-\frac{Nt^2}{c} \right ).
    \end{split}
\end{equation}
\end{thm}

Let us elaborate more on the two terms in the deviation bound (\ref{concentration}). The first term $\tau^2(\theta_*,p_0)$ can be thought of as a `bias' term whose size is determined by the sum of the Hellinger distance $h^2(S_{\theta_*},p_0)$ and penalty norm $\textnormal{pen}(\theta_*)^2$, of some `best approximating element' $\theta_*$. In contrast, the minimal size of the second term $t^2\ge \delta^2$ is determined by the lower bound (\ref{delta-req}) which encapsulates the metric entropy (and thus statistical complexity) of the classes of transport maps $\mathcal S_*(\lambda, R)$ which one optimizes over. This intuition is reminiscent of previous classical results in M-estimation \cite{VDG00}. We note that the $|\cdot|$-norm covering bounds are required in order to ultimately control the Hellinger-norm bracketing numbers for the pushforward densities, which in turn control the size of the fluctuations of the empirical process in our proofs; see Section \ref{gen-pfs} and especially Lemma \ref{lem-fwd} below for details.

The preceding theorem implies the following convergence rate in Kullback--Leibler divergence (see also Corollary \ref{cor-MLE} above).

\begin{cor}
\label{thm-KL}
Suppose the assumptions of Theorem \ref{thm-penalty-gen} hold, and suppose in addition that $p_0$ is lower bounded by some constant $l>0$. Then there exists some $c'>0$ such that for all $\theta_*, \delta , N, \delta$ as in Theorem \ref{thm-penalty-gen},
\[ P_0^N\left ( \textnormal{KL}(p_0, S_{\hat \theta_N}^\#\eta \big) \ge c'[\tau^2(\theta_*,p_0) +t^2]^2 \right ) \le c'\exp\left (-\frac{Nt^2}{c'}\right ).\]
\end{cor}

\section{Triangular maps II: parameterizations and penalties}\label{triangle-II}

We now continue the study of estimators based on the triangular Knothe--Rosenblatt maps introduced in Section \ref{MLE-res}. Equipped with the general theory developed in Theorem \ref{thm-penalty-gen}, we may now study more general classes of estimators than the (somewhat restrictive) unpenalized maximum likelihood estimators from Section \ref{MLE-res}. As noted earlier, it is \textit{a priori} rather unclear how to parameterize the class $\KR(\alpha,L,c_{\text{mon}})$ of monotone triangular maps with bounded anisotropic H\"older norm in \eqref{KR}, and thus how to computationally realize the estimator in Theorem~\ref{thm-anisotropic}. Even setting aside questions of how to impose the desired regularity, care is required to parameterize the convex cone of monotone triangular maps from $Q_d$ to itself, due to the monotonicity and range constraints at hand. 
%
%
To surmount these issues, we propose a `rational parameterization' of triangular maps such that minimization of the relevant penalized objective (\ref{JN}) becomes an unconstrained problem. Subsequently, in Sections \ref{sec:sobolev} and \ref{sec:wavelets}, we study two concrete estimators arising as penalized MLEs with a Tikhonov-type Sobolev penalty and a wavelet based penalty, respectively.

\subsection{Rational parameterization of triangular maps on $Q_d$}

Roughly speaking, two key properties characterize a valid component function $S_k$ of the triangular KR map \eqref{KR}: first the \textit{monotonicity constraint} that $S_k$ be monotone increasing in the $k$-th variable, and second the \textit{range constraint} that $S_k$ map $Q_k$ onto $[0,1]$, for each $k=1,\dots,d$.

To this end, we introduce the following `rational' parameterization of triangular transports, where the parameter is a multivariate function $F:Q_d\to \R^d$. For some fixed `link function' $\Phi: \R\to (0,\infty)$ and bounded function $F_k:Q_k\to \R$ for $k=1,\dots,d$, 
let
\begin{equation}\label{S-F}
    S_{F,k}(x_{1:k})\coloneqq \frac{\int_{0}^{x_k}\Phi \big(F_k(x_{1:k-1},y) \big)dy}{\int_{0}^{1}\Phi \big(F_k(x_{1:k-1},y) \big)dy},~~~~x_{1:k}\in Q_k.
\end{equation}
It is straightforward to see that for such $F=(F_1,\dots,F_d)$, $S_F=(S_{F,1},\dots, S_{F,d}):Q_d\to Q_d$ forms a Knothe--Rosenblatt map. 

While the construction (\ref{S-F}) is admissible for any (say, continuously differentiable) link function $\Phi>0$ (see for instance \cite{ZM20}, where $\Phi= (x+1)^2$ is used), it will be convenient for our theory developed below to make the following regularity assumptions on $\Phi$. The upper and lower bounds on $\Phi$ in Definition \ref{reg-link} below will serve to verify the uniform boundedness assumptions (\ref{C1-diag-bd})--(\ref{general-lb}) for the resulting classes $\mathcal S$ of transport maps. 

\begin{defin}[Regular link function]\label{reg-link}
    For constants $0<K_{min}<K_{max}<\infty$, we assume that $\Phi$ is a smooth, bijective and strictly increasing map 
    \begin{equation}\label{phi-req}
        \Phi:\R\to (K_{min},K_{max}),~~~~\Phi'>0.
    \end{equation}
\end{defin}
For example, $\Phi$ could be a `logistic' function of the form $\Phi(x)=K_{min}+(K_{max}-K_{min})/(1+e^{-x})$.


\subsection{Example I: Sobolev-type penalty}
\label{sec:sobolev}

We now utilize (\ref{S-F}) to re-parameterize KR maps by $F$ belonging to some sufficiently regular function space. For the sake of conciseness we focus here on $L^2$-Sobolev spaces, which due to their Hilbert space (i.e., Euclidean) structure are most convenient for the purposes of numerical computation (see the end of this subsection for discussion of possible generalizations).

For $k\ge 1$, let $H^\alpha(Q_k)$ be the space of real-valued square-integrable functions on $Q_k$ with square-integrable weak derivatives of up to order $\alpha$, normed by
\[ \|F\|_{H^\alpha(Q_k)}^2=\sum_{|\beta|\le \alpha} \|\partial_{\beta}F \|_{L^2(Q_k)}^2,\]
and denote its vector valued counterpart by $H^\alpha(Q_k,\R^d),~d\ge 1$.
We moreover define the space of triangular Sobolev functions as
\begin{align*}
    H^\alpha_\Delta(Q_d,\R^d)=\big\{ & F\in H^\alpha(Q_d,\R^d):\\ & F(x)=(F_1(x_1),\dots,F_d(x_{1:d}) )~ \textnormal{for some}~ F_k\in H^\alpha(Q_k),~1\le k\le d\big\},
\end{align*} 
with norm
\[\|F\|_{H^\alpha_\Delta(Q_d,\R^d)}^2=\sum_{k=1}^d\|F_{k}\|_{H^\alpha(Q_k)}^2.\]

For some integer $\alpha>d/2+1$, set the parameter space to be $H^\alpha_\Delta(Q_d,\R^d)$. For some regularization parameter $\lambda_N>0$, the objective to be minimized then reads \ymm{added missing $\eta$. it goes missing in other places too (we should check around on final proofread)}
\begin{equation*}
    \begin{split}
        \mathcal J_{N}:H^\alpha_\Delta(Q_d,\R^d)\to \R,~~~ \mathcal J_{N}(F) \coloneqq -\frac 1N\sum_{i=1}^N\log S_F^\# \eta(X_i) + \lambda_N^2 \|F\|^2_{H^\alpha_\Delta(Q_d,\R^d)}.
    \end{split}
\end{equation*}
Note that since $H^\alpha(Q_k)\subseteq C^{1}(Q_k)$ for all $1\le k\le d$ by the standard Sobolev embedding, the pullback density $S_F^\# \eta(X_i)$ is well-defined. Let $\hat F_N$ denote any (measurable) choice of minimizer
\[ \hat F_N \in {\arg\min}_{F\in H^\alpha_\Delta(Q_d,\R^d)} \mathcal J_N(F).\]

\begin{thm}\label{thm-sobolev}
    Let $\alpha > d/2 + 1$ be some integer and suppose that $p_0\in \mathcal M(\alpha,B,c)$ for some $0<c<B<\infty$. Set $\lambda_N =N^{-\frac{\alpha}{2\alpha+d}}$. Assume moreover that $\eta\in \tilde {\mathcal M}(\alpha,B,c)$ from (\ref{M-fact}). Then there are constants $0<K_{min}<K_{max}< \infty$ such that for any link function $\Phi$ satisfying (\ref{phi-req}), there exists $c'>0$ (which may be chosen uniformly over $p_0,\eta$ as above)\ssw{Add this to other thms too?} so that for any $M\ge c'$ and $N\ge 1$,
    \[ P_{0}^N\left ( h^2(S_{\hat F_N}^\#\eta, p_0) + \lambda_N^2 \|F\|^2_{H^\alpha_\Delta(Q_d,\R^d)} \ge M^2N^ {-\frac{2\alpha}{2\alpha+d}}\right ) \le c'\exp\left (- \frac{M^2N^{d/(2\alpha+d)}}{c'}\right ).\]
\end{thm}
The proof of Theorem \ref{thm-sobolev}, which can be found in Section \ref{sec-triII-proofs}, is based on an application of our general theory from Section \ref{sec-penalty}, with $\Theta= H_\Delta^\alpha(Q_k,\R^d)$. The $L^2$-type Sobolev penalty chosen here resembles typical Tikhonov regularisation schemes which are commonly used in inverse problems \cite{EHN96,BB18}; statistical convergence rates for such estimators have for instance been studied in \cite{NVW18}, using techniques similar to the ones employed here.

We also note that the choice of $L^2$-type penalty norm here was made for convenience. In general, other $\alpha$-degree (H\"older or Besov) smoothness penalties would lead to the same convergence results; see, e.g., \cite{AW21}, where penalized least squares estimators with Besov penalty are studied in an inverse problems context. However, the choice of $L^2$-penalty is made most commonly in practical situations due to strict convexity and the computationally convenient Hilbert space structure; we thus decided to focus our efforts on this example.

\subsection{Example II: Wavelet-based penalty}
\label{sec:wavelets}

In practice, one employs high-dimensional discretization schemes to approximate infinite-dimensional regularity classes, and wavelet truncations are one of the most commonly used methods to do so. For any $1\le k\le d$, let \[\Psi^k=\{\psi_{lm}^k:l\ge 0, m\in \Z^k\} \]
denote the standard collection of `$S$-regular,' compactly supported Daubechies tensor wavelets forming an orthonormal basis of $L^2(\R^k)$; see, e.g., \cite{D92, M92} for definitions ($S>\alpha$ will be chosen later to be sufficiently large). Here, $l\ge 0$ denotes the resolution level, and $m$ is the index of wavelets at each resolution level. For convenience (and in slight abuse of notation) we also re-enumerate the wavelets in $\Psi^k$ \ymm{superscript for consistency} at at each level $l$ by a single-integer index $m\ge 1$. 
Since we wish to model functions on the cube $Q_k$, we consider only those wavelets whose support has a non-empty intersection with $Q_k$,
\[ \Psi^k_{Q_k} \coloneqq \big\{\psi\in \Psi^k: \text{supp}(\psi)\cap Q_k \neq \emptyset\big\}. \]
Let $L_l^k\ge 1$ denote the number of wavelets in $\Psi^k_{Q_k}$ at level $l$, satisfying $L_l^k\simeq 2^{lk}$. Without loss of generality, let us denote the indices of those wavelets by $m\in \{1,\dots,L_l^k\}$. For $J\ge 1$, we denote the index set up to resolution level $J$ by
\[ I_{k,J}\coloneqq \big\{(l,m): 0\le l\le J,~1\le m\le L^k_l \big\}. \]
We will use the following parameterization of triangular transport maps using wavelets.

\begin{defin}[Wavelet-based parameterization of KR maps]
    Let $0<K_{min}<K_{max}<\infty$, and suppose $\Phi:\R\to (K_{min},K_{max})$ is a regular link function in the sense of Definition \ref{reg-link}. Let $J\ge 1$ be some resolution level. Define the parameter space $\Theta$ as 
    \begin{align}\label{theta-def}
        \Theta \coloneqq \bigotimes_{k=1}^d \R^{I_{k,J}}=\big\{\theta=(\theta_1,\dots,\theta_d) \ \textnormal{s.t.}\ \theta^k=\big(\theta_{lm}^k: (l,m)\in I_{k,J}\big) \in \R^{I_{k,J}},~1\le k\le d \big\}.
    \end{align}
    Given any $\theta\in \Theta$, we then define the $k$-th component $S_{\theta,k}$ of the KR map $S_\theta :Q_d\to Q_d$ by 
    \begin{equation}\label{S-theta}
        S_{\theta,k}(x_{1:k})\coloneqq \frac{\int_{0}^{x_k}\Phi \big(F_{\theta,k}(x_{1:k-1},y) \big)dy}{\int_{0}^{1}\Phi \big(F_{\theta,k}(x_{1:k-1},y) \big)dy},~~~~x_{1:k}\in Q_k,
    \end{equation}
    where $F_{\theta,k}\in L^2(Q_k)$ is given by the wavelet expansion
    \begin{equation}\label{F-theta}
        F_{\theta,k}(x_{1:k})\coloneqq\sum_{(l,m)\in I_{k,J}}\theta_{lm}^k\psi_{lm}^k(x_{1:k}).
    \end{equation}
\end{defin}
For $\alpha>d/2+1$, we now equip the space $\Theta$ with the wavelet-type sequence norm
\[ \|\theta\|_{b^\alpha_{22}}^2\coloneqq\sum_{k=1}^d\sum_{l=0}^J\sum_{m=1}^{L_l^k} 2^{2l\alpha }|\theta^k_{lm}|^2,~~~ \theta\in \Theta. \]
Again, other choices would have been feasible here; however we focus on the case of (computationally convenient) $\ell^2$-type penalties, yielding the objective function
\begin{equation}\label{wave-MLE}
    \mathcal J_N:\Theta = \bigotimes_{k=1}^d\R^{I_{k,J}} \to \R,~~~\mathcal J_N(\theta) \coloneqq-\sum_{i=1}^N \log S_\theta^\# \eta (X_i)+ \lambda_N^2 \|\theta\|_{b^\alpha_{22}}^2,
\end{equation}
where $\lambda_N>0,~J\ge 1$ are tuning parameters to be chosen below. Let $\hat \theta_N$ denote any (measurable with respect to $X_1,\dots, X_N$) choice $\hat \theta_N \in \arg\min_{\theta\in\Theta} \mathcal J_N$.


\begin{thm}\label{thm-wave}
    Let $\alpha > d/2 + 1$ be some integer and assume that $p_0\in \mathcal M(\alpha,B,c)$ and $\eta\in \tilde {\mathcal M}(\alpha,B,c)$, for some $0<c<B<\infty$. Set $\lambda_N \coloneqq N^{-\frac{\alpha}{2\alpha+d}}$, and set the resolution level $J=J_N$ such that $2^{Jd}\simeq N\delta_N^2=N^{\frac{d}{2\alpha +d}}$. Then there are constants $0<K_{min}<K_{max}< \infty$ such that for any regular link function $\Phi$ (in the sense of Definition \ref{reg-link}), 
    there exists $c'>0$ so that for any $M\ge c'$ and $N\ge 1$,
    \[ P_0^N\left ( h^2\big(S_{\hat \theta_N}^\#\eta,p_0\big) + \lambda_N^2 \|\hat \theta_N\|_{b^\alpha_{22}}^2 \ge M^2N^{-\frac{2\alpha}{2\alpha+d}} \right )\le c'\exp\left ( -\frac {M^2N^{d/(2\alpha+d)}}{c'}  \right). \]
\end{thm}

\subsection{Estimation of the transport map}

So far, we have primarily been concerned with the estimation of the data-generating density $p_0$ as the pushforward of some learned transport map. However, one may also be interested in the recovery of the (unique up to coordinate orderings) KR rearrangement $S_{p_0,\eta}$. We achieve this by employing the following \textit{stability estimate}, bounding the modulus of continuity for the nonlinear map $p_0\mapsto S_{p_0,\eta}$ (given a fixed reference $\eta$).
\begin{lem}\label{lem-stability}
    For some $c>0$, let $\eta:Q_d\to [c,\infty)$ be a fixed $C^1$ reference density which factorizes into $C^1$ one-dimensional marginals $\eta(x)=\Pi_{k=1}^d e_k(x_k)$. Then, for any constants $0< m< M<\infty$ there exists $C>0$ such that for any Lipschitz densities $p,\bar p \ge m$ with Lipschitz norm $\max\{\|p\|_{Lip},\|\bar p\|_{Lip}\}\le M$, we have that
    \[ \|S_{p,\eta}-S_{\bar p, \eta}\|_{H^1_{\text{diag}}}\le C \|p-\bar p\|_{L^2}.\]
\end{lem}

Combining the preceding stability estimate with our previous theorems, one immediately obtains a convergence rate for the reconstruction error of $\hat S_N$. In contrast to recent work in the context of optimal transport \cite{HR20,MBNW21}, the loss function obtained here is the natural anisotropic $\|\cdot\|_{H^1_{\text{diag}}}$-norm corresponding to the Hellinger distance between $p_0$ and $\hat S_N^\# \eta$ (as opposed to $\|\cdot\|_{L^2}$); see also the discussion at the end of Section \ref{sec:intro}.

\begin{cor}\label{cor-map}
Let $\hat S_N$ denote either (i) the estimator $S_{\hat F_N}$ from Theorem \ref{thm-sobolev}, (ii) the estimator $S_{\hat \theta_N}$ from Theorem \ref{thm-wave}, or (iii) $\hat S$ from Theorem \ref{thm-anisotropic}; and assume the hypotheses of the respective theorems.
Then, we have that
\[ E_{P_0^N}\big\|\hat S_N - S_{p_0,\eta}\big\|_{H^1_{\text{diag}}}^2 \lesssim N^{-\frac{2\alpha}{2\alpha+d}}. \]
\end{cor}

\section{Proofs for Section \ref{sec-penalty}}\label{gen-pfs} The proof of Theorem \ref{thm-penalty-gen} adapts techniques developed in \cite{VDG00,VDG01} (see also Chapter 7 in \cite{GN16}) to the current transport-based setting.

\begin{proof}[Proof of Theorem \ref{thm-penalty-gen}] 

1. \textit{`Basic inequality.'} For ease of notation, let us write $\hat \theta=\hat\theta_N$, $p^*\coloneqq S_{\theta_*}^\#\eta$, $\hat p \coloneqq S_{\hat \theta}^\#\eta$ as well as $I(\theta)\coloneqq\textnormal{pen}(\theta)$. We begin by deriving an appropriate `basic inequality.' Observe that

\[ \int \log \frac{\hat p}{p^*}dP_N - \lambda_N^2I^2(\hat \theta) + \lambda_N^2I^2(\theta_*) \ge 0. \]
By the concavity of $\log$ and using that $\frac 12 \log \frac{\hat p}{p^*}= \frac 12 (\log \hat p+ \log p^*)- \log p^*$ we obtain
\begin{equation}\label{basic-start}
    2\int \log \frac{\hat p+p^*}{2p^*}dP_N - \lambda_N^2I^2(\hat \theta) \ge -\lambda_N^2I^2(\theta_*).
\end{equation}
Now, using the assumptions (\ref{C1-bd})--(\ref{general-lb}), we see that for any choice of $\theta_*$,
\begin{equation}\label{p*-lb}
    \inf_{x\in U} p^*(x) = \inf_{x\in U} \eta(S(x))\det \nabla S(x) \ge \tilde K^{-1}L^{-1}.
\end{equation}
Using this as well as $p_0\le K$, it follows that 
\begin{equation}\label{ratio-bound}
    \sup_{x\in U}\frac{p_0}{p^*}\le K \tilde K L.
\end{equation}
Therefore, writing $C\coloneqq2\big(1+\sqrt
{K\tilde KL}\big)$ and arguing similarly as in pages 189-191 of \cite{VDG00}, we obtain the following inequality for the empirical process on the left hand side of (\ref{basic-start}):
\begin{equation*}
\begin{split}
\frac 12 \int \log \frac{\hat p+p^*}{2p^*}dP_N &\le \frac 12 \int \log \frac{\hat p+p^*}{2p^*}d(P_N-P)- \int \Big( 1-\sqrt{\frac{\hat p+p^*}{2p^*}}\Big) dP\\
&\le \frac 12 \int \log \frac{\hat p+p^*}{2p^*}d(P_N-P) - h^2\Big( \frac{\hat p+p^*}{2},p^* \Big) \\
&~~~~~~~~~~~~~ +2\Big(1+\sup_{x\in U}\sqrt{\frac{p_0}{p^*}}\Big) h\Big(\frac{\hat p+p^*}{2},p^* \Big)h(p^*,p_0),\\
&\le \frac 12 \int \log \frac{\hat p+p^*}{2p^*}d(P_N-P) - h^2\Big( \frac{\hat p+p^*}{2},p^* \Big) \\
&~~~~~~~~~~~~~ +C h\Big(\frac{\hat p+p^*}{2},p^* \Big)h(p^*,p_0).
\end{split}
\end{equation*}
Combining this with (\ref{basic-start}), we obtain
\begin{equation}\label{splitting}
    \begin{split}
        h^2\Big( \frac{\hat p+p^*}{2},p^* \Big)+ \lambda^2 I^2(\hat \theta)&\lesssim \int \log \frac{\hat p+p^*}{2p^*}d(P_N-P) \\
        &~~~~~~~~~~~~~~~~~~~~~+ h\Big(\frac{\hat p+p^*}{2},p^* \Big)h(p^*,p_0) + \lambda^2 I^2(\theta_*) \\
        &=: I+II+III.
    \end{split}
\end{equation}
Since $h(\frac{\hat p+p^*}{2},p^*)\le 4h(\hat p,p^*)$, note that the l.h.s.\ in this inequality is an upper bound for $\tau^2(\hat\theta_N,p^*)/4$. Now, we make the following case distinction.

\smallskip

\textit{Case 1: $III \ge \max\{I,II\}$.} Then we immediately obtain that $\tau^2(\hat \theta_N,p^*)\lesssim \tau^2(\theta_*,p_0)$.

\smallskip

\textit{Case 2: $II \ge \max\{I,III\}$.} In this case, we have that \[h^2\Big( \frac{\hat p+p^*}{2},p^* \Big)+ \lambda^2 I^2(\hat \theta) \lesssim  h\Big(\frac{\hat p+p^*}{2},p^* \Big)h(p^*,p_0). \]
Dividing this by $h(p^*,p_0)$, we also obtain
$h\big(\frac{\hat p+p^*}{2},p^* \big)\lesssim h(p^*,p_0)$. Using (\ref{splitting}) it follows that $\tau^2(\hat \theta_N,p^*)\lesssim \tau^2(\theta_*,p_0)$.

\smallskip

\textit{Case 3: $I\ge \max\{II,III\}$.} This is the most important case, and by what precedes we may restrict to this event for the remainder of the proof. Writing shorthand $g^*(p)= \log ((p+p^*)/2p^*)$ for any density $p$, we obtain here the `basic inequality'
\begin{equation}\label{basic}
    \tau^2(\hat \theta_N,p^*) \lesssim \int g^*(\hat p)d(P_N-P).
\end{equation}

\smallskip

2. \textit{`Slicing' argument}. Suppose that (\ref{basic}) holds. Then, for some constant $c_1>0$ and any $t>0$ we can estimate
\begin{equation*}
\begin{split}
    P_0^N&\big(\tau^2(\hat \theta,p^*) \ge t\big) \\
    &= \sum_{s=0}^\infty P_0^N\Big[\tau^2(\hat \theta,p^*) \in [2^{s}t,2^{s+1}t), \int g^*(\hat p)d(P_N-P_0) \ge c_12^st\Big]\\
    &\le \sum_{s=0}^\infty P_0^N\Big[ \sup_{\theta:\tau^2 ( \theta,p^*)\le 2^{s+1}t} \Big|\int g^*(\hat p) d(P_N-P_0)\Big| \ge c_12^st \Big].
\end{split}
\end{equation*}
We introduce the notations
\begin{equation}\label{star-sets}
    \begin{split}
    \Theta_s^*\coloneqq&\big\{\theta:\tau^2( \theta,p^*)\le 2^{s+1}t \big\},\\
    \mathcal G_s^*\coloneqq&\big\{g_{\theta}^*:\tau^2( \theta,p^*)\le 2^{s+1}t \big\},\\
    \mathcal S_s^*\coloneqq&\big\{S_\theta:\tau^2( \theta,p^*)\le 2^{s+1}t \big\}.
    \end{split}
\end{equation}
Then we can bound the latter sum by
\begin{equation}\label{slices}
    \sum_{s=0}^\infty P_0^N\Big[ \sup_{g\in \mathcal G_s^*}  \Big|\sqrt N \int g d(P_N-P_0)\Big| \ge \sqrt {c_1N} 2^{s}t \Big].
\end{equation}

3. \textit{Concentration via bracketing entropy rates}.
To control each term in this sum, we will employ concentration of measure techniques for empirical processes; specifically, we aim to apply Theorem 5.11 in \cite{VDG00}. To this end, we define the `Bernstein size' (see, e.g., p.~206 of \cite{GN16}, with $K=1$) of any function $f:U\to \R$ by
\begin{equation}\label{rho-1}
    \rho_1(f)=E_{X\sim P_0 }\big[e^{|f(X)|}-|f(X)|-1\big].
\end{equation} 
Then, by Lemma \ref{lem-bernstein} and the definition of $\mathcal G_s^*$, we have that for some constants $c_2,c_3>0$,
\begin{equation}\label{RS}
\begin{split}
    \sup_{g\in \mathcal G_s^*} \rho_1(g) &\le c_2 \sup_{\theta\in \Theta_s^*} h\Big(\frac{S_\theta^\#\eta+ p^*}2,p^*\Big)\le c_3 \sup_{\theta\in \Theta_s^*} h(S_\theta^\#\eta,p^*)\\
    &\le c_3 \sup_{\theta\in \Theta_s^*} \tau(\theta,p^*)
    \le c_3 2^{(s+1)/2}t^{1/2} =:R_s.
\end{split}
\end{equation}

We now need some standard notation for bracketing entropy numbers. For a collection $\mathcal F$ of functions $U \to \R$ and some Borel measure $\nu$ on $U$, we denote by $N_B(\rho, \mathcal F, \nu)$ the $L^2(\nu)$ bracketing number of $\mathcal F$, i.e., the minimal number of brackets $(f_L^{(i)}, f_U^{(i)})\in \mathcal F^2$, $0\le f_L^{(i)}\le f_U^{(i)}$, of $L^2$-size $\|f_L^{(i)}-f_U^{(i)}\|_{L^2(\nu)}\le \rho$
such that for each $f\in\mathcal F$ there exists $i\in \{1,\dots,N_B(\rho, \mathcal F, \nu)\}$ with $f_L^{(i)}\le f\le f_U^{(i)}$.
The ($L^2$-)bracketing entropy of $\mathcal F$ is
\begin{equation}\label{bracketing}
    H_B(\mathcal F, \nu,\rho) = \log N_B(\mathcal F,\nu, \rho),~~~ \rho>0.
\end{equation}
In addition to this, we also require the Bernstein-size bracketing metric entropy $H_{B,\rho_1}(\mathcal F,P_0,\eps)$ where the $L^2$-norm 
above is replaced by $\rho_1$ from (\ref{rho-1}); namely
\[ H_{B,\rho_1}(\mathcal F, \nu,\eps) = \log N_{B,\rho_1}(\mathcal F,\nu,\eps),~~~ \eps>0,\]
where $N_{B,\rho_1}(\mathcal F,\nu,\eps)$ is the minimal number of brackets $(f_L^{(i)}, f_U^{(i)})$ required to cover $\mathcal F$ such that $\rho_1(f_U^{(i)}-f_L^{(i)})\le \eps$ (cf.\ e.g., Definition 3.5.20 of \cite{GN16}).

To apply Theorem 5.11 in \cite{VDG00} to (\ref{slices}), we require bounds for $H_{B,\rho_1}(\mathcal G_s^*,P_0,\eps)$, $\eps>0$. Using Lemma \ref{lem-bernstein} below, we may initially bound this by the $L^2$-bracketing metric entropy of the corresponding densities,
\[ H_{B,\rho_1}(\mathcal G_s^*,P_0,\eps)\le H_B\Big(\Big\{\sqrt{\frac{S_\theta^\#\eta+p^*}{2}}:\theta\in \Theta_s \Big\},\|\cdot\|_{L^2(P_0)},\frac \eps C\Big),~~~\eps>0, \]
for some $C>0$. We now further estimate these entropies in terms of entropies of classes of transport maps. For clarity, let us assume the case of Assumption \ref{generic-assumption} \textbf{ii)}, i.e., the case of triangular maps. (The case of Assumption \ref{generic-assumption} \textbf{i)} follows in exactly the same way, by employing (\ref{C1-lip}) instead of (\ref{pushfwd-cont}) in Lemma \ref{lem-fwd} below.) By (\ref{p*-lb}), $p^*$ is uniformly bounded away from $0$. Using this, Lemma \ref{lem-fwd}, as well as the assumption (\ref{C1-diag-bd}) (which implies that the maximum on the right hand side of (\ref{pushfwd-cont}) is uniformly bounded by a constant only depending on $\Phi$), it follows that for some $C',C''>0$, 
\begin{equation*}
\begin{split}
    H_{B,\rho_1}(\mathcal G_s^*,P_0,\eps)&\le H_B(\{S_\theta^\#\eta:\theta\in\Theta_s^*\},P_0,\eps/C')\\
    &\le H(\{S_\theta^\#\eta:\theta\in\Theta_s^*\},\|\cdot\|_\infty,\eps/C')\\
    &\le H(\mathcal S_s^*,\|\cdot\|_{C^1_{\text{diag}}},\eps/C'').
\end{split}
\end{equation*}
Since $t\ge \delta$, (\ref{delta-req})\ssw{Check again that we can actually choose constant 1 in (3.10)?} and the fact that $\Psi_*(\lambda,R)/R^2$ is decreasing imply that
\begin{equation*}
\begin{split}
    \int_0^{R_s} H_{B,\rho_1}^{1/2}(\mathcal G_s^*,P_0,\eps) d\eps &\lesssim \Psi_*(\lambda,R_s)\lesssim \sqrt N R_s^2.
\end{split}
\end{equation*}
Using this and (\ref{RS}), we may in particular choose $C_0 >0$ (small enough) and $C_1>0$ (large enough) such that for all $N\ge 1$, 
\begin{align*}
    \sqrt{c_1N} 2^{s-1}t & \le C_1 \sqrt N R_s^2,~~~\text{as well as}~~~ \\
    \sqrt{c_1N} 2^{s+1}t & \ge C_0 \Big(\int_0^{R_s} H_{B,\rho_1}^{1/2}(\mathcal G_s^*,P_0,\eps) d\eps \Big). &
\end{align*}
This verifies the hypotheses (5.30), (5.31), and (5.33) in Theorem 5.11 of \cite{VDG00}, with choices $K=1$, $a=\sqrt {c_1N} 2^{s}t$ as well as $R=R_s$. Noting that we may drop the requirement (5.32)---see the discussion on p.76 of \cite{VDG00}---an application of Theorem 5.11 of \cite{VDG00} to (\ref{slices}) yields that for some large enough constants $c_4,c>0$,
\begin{equation*}
    P_0^N\big(\tau^2(\hat\theta,p^*) \ge t\big) \le c_4\sum_{s=0}^\infty \exp\Big\{-\frac{2^{2s}Nt}{c_4}\Big\} \le c \exp\{Nt/c\big\}.
\end{equation*}
Since
\[ \tau^2(\hat\theta,p_0)\le 2[\tau^2(\hat\theta,p^*)+h^2(p^*,p_0)] \le 2[\tau^2(\hat\theta,p^*)+\tau^2(\theta_*,p_0)],\]
we have
\begin{equation*}
\begin{split}
    P_0^N\big(\tau^2(\hat\theta,p_0) \ge 2[\tau^2(\theta_*,p_0) +t^2] \big)&\le P_0^N\big(2[\tau^2(\hat\theta,p^*)+\tau^2(\theta_*,p_0)] \ge 2[\tau^2(\theta_*,p_0) +t^2] \big)\\
    &\le P_0^N\big(\tau^2(\hat \theta,p^*) \ge t^2 \big),
\end{split}
\end{equation*}
which completes the proof.
\end{proof}

The following lemma is an adaptation of Lemmas 7.2 and 7.3 in \cite{VDG00} to our setting.

\begin{lem}\label{lem-bernstein}
    Suppose the hypotheses of Theorem \ref{thm-penalty-gen} hold.
    \\ \textbf{i)} There exists some $C>0$ such that for any $\theta_* \in \Theta$ and probability density $p:U\to \R$, it holds that \[\rho_1(g^*(p))\le C h\Big(\frac{p+p^*}{2},p^*\Big).\]
    \textbf{ii)} There exists some $C>0$ such that for any $\theta_* \in \Theta$ and any $0< p_L\le p_U$,
    \begin{equation*}
        \rho_1(g^*(p_U)-g^*(p_L))\le Ch\Big( \frac{p_L+p^*}{2}, \frac{p_U+p^*}{2}\Big).
    \end{equation*}
\end{lem}
\begin{proof}
    \textbf{i)} Using Lemma 7.1 in \cite{VDG00} as well as (\ref{ratio-bound}), we have that 
    \begin{align*}
        \rho_1^2(g^*(p))&\le 8 \int (e^{g^*_p}-1)^2 dP_0 =  8 \int \Big(\sqrt{\frac{p+p^*}{2p^*}}-1\Big)^2  \frac{p_0}{p^*} dP^*\\
        &\le 8K\tilde KL\cdot h^2\Big(\frac{p+p^*}{2} ,p^*\Big).
    \end{align*}
    \textbf{ii)}
    Using again Lemma 7.1 in \cite{VDG00} as well as (\ref{ratio-bound}), we obtain
    \begin{equation*}
    \begin{split}
        \rho_1^2(g^*(p_U)-g^*(p_L))&\le \int \Big(\sqrt{\frac{p_U+p^*}{p_L+p^*}}-1 \Big)^2dP_0\\
        &= \int \Big(\sqrt{\frac{p_U+p^*}{2}}-\sqrt{\frac{p_L+p^*}{2}} \Big)^2 \frac{2p^*}{p_L+p^*} \frac{p_0}{p^*} d\lambda_U\\
        &\le 2K\tilde K L \cdot h\Big( \frac{p_L+p^*}{2}, \frac{p_U+p^*}{2}\Big).
    \end{split}
    \end{equation*}
\end{proof}

The following simple local Lipschitz estimate for the pullback operation $S\mapsto S^\#\eta$ is crucial for the proof of Theorem \ref{thm-penalty-gen} above.

\begin{lem}\label{lem-fwd}
    Let $U,V\subseteq \R^d$ be bounded domains and let $\eta: V\to \R$ be a bounded, Lipschitz continuous reference density. Then there exists some constant $C=C(\eta,U)>0$ and some increasing function $M:(0,\infty)\to (0,\infty)$ such that for any two diffeomorphisms $S,\tilde S:U\to V$, we have that
    \begin{align}\label{C1-lip}
        \|S^\#\eta - \tilde S^\#\eta \|_\infty \le C\cdot M(\|S\|_{C^1} \vee \|\tilde S\|_{C^1}) \|S-\tilde S\|_{C^1} .
    \end{align}
    If additionally $U,V$ are convex, then there exists $C>0$ such that for any triangular maps $S,\tilde S:U\to V$,
        \begin{align}\label{pushfwd-cont}
        \|S^\#\eta - \tilde S^\#\eta \|_\infty \le C \|S-\tilde S\|_{C^1_{\text{diag}}}\max \big\{\|S\|_{C^1_{\text{diag}}}^{d-1},\|\tilde S\|_{C^1_{\text{diag}}}^{d-1}\big\}.
    \end{align}
\end{lem}

\begin{proof}
    Let us write $p=S^\#\eta$, $\tilde p= \tilde S^\#\eta$, and suppose that $S,\tilde S$ are triangular. Then, using the boundedness and Lipschitz continuity of $\eta$, we obtain for any $x\in U$ that
    \begin{equation*}
    \begin{split}
        |p(x)-\bar p(x)|&=\big|\eta(S(x)) \det \nabla S(x) - \eta(\bar S(x)) \det \nabla \bar S(x) \big|\\
        &\le \|\eta\|_{Lip}\|S\|_{C^1_{\text{diag}}}^d \|S-\bar S\|_\infty \\
        &~~~~~~~~~~~~~~~~~~~~ + d\|\eta\|_\infty \max \Big\{\|S\|_{C^1_{\text{diag}}}^{d-1},\|\bar S\|_{C^1_{\text{diag}}}^{d-1}\Big\} \|S-\bar S\|_{C^1_{\text{diag}}}.
    \end{split}
    \end{equation*}
    The corresponding statement (\ref{C1-lip}) for maps that are not necessarily triangular follows along the same lines, and its proof is thus omitted.
\end{proof}

We conclude this section with the proof of the KL-convergence rate from Corollary \ref{thm-KL}.

\begin{proof}[Proof of Corollary \ref{thm-KL}]

By Assumption \ref{generic-assumption}, we have that
\[ \sup_{S\in \mathcal S}\|S^\# \eta \|_\infty \le \sup_{S\in \mathcal S}\| (\eta\circ S) \det \nabla S\|_\infty =: M \le \infty.\]
Moreover, Taylor expanding the logarithm, for any lower bounded density $q\ge l>0$  and some intermediate value 
\[ \xi(x) \in \big[\min (l, L^{-1}), \max (K, \sup_{S\in\mathcal S} \|S_\theta^\# \eta\|_\infty)\big], \]
it holds that
\begin{equation*}
    \begin{split}
    \textnormal{KL}(p_0, q\big) &= -\int_{U} (\log q(x)- \log p_0(x) ) p_0(x) dx \\
    &= - \int_{U} \frac{p_0(x)}{p_0(x)} ( q(x)- p_0(x) )  dx + \int_{U} \frac{p_0(x)}{2 \xi(x)^2} ( q(x)- p_0(x) )^2 dx\\
    &\le K \max\{l^{-1},L\}^2 \int_{Q_d} ( q(x)- p_0(x) )^2 dx  \lesssim \|q-p_0\|_{L^2}^2.
    \end{split}
\end{equation*}
Since the densities $q\in \{ S^\#\eta: S\in\mathcal S \}$ are uniformly lower bounded, we may in particular assert that  $\|S^\#_{\hat \theta_N}\eta -p_0\|_{L^2}^2\lesssim h^2(S^\#_{\hat \theta_N}\eta ,p_0)\le \tau^2(\hat \theta_N,p_0)$, whence the desired result follows from an application of Theorem \ref{thm-penalty-gen}.
\end{proof}

\section{Proofs for Section \ref{triangle-II}}\label{sec-triII-proofs}

We begin with the proof of the Theorem \ref{thm-sobolev} for Sobolev-type penalties.

\begin{proof}[Proof of Theorem \ref{thm-sobolev}]
Fix any $p_0\in \mathcal M(\alpha,B,c)$. There exists a unique triangular and monotonically increasing transport $S^{p_0}$ satisfying $S^{p_0}_\#p_0=\eta$, as constructed in Section \ref{MLE-res}. By the smoothness assumptions on $p_0,\eta$ as well as Proposition \ref{prop-fact}, we have that $S^{p_0}\in \KR(\alpha,L,c_{\text{mon}})$ for some $\alpha,L,c_{\text{mon}}$ only depending on $\alpha,B,c$.

Since the parameterization $F\mapsto S_F$ from (\ref{S-F}) is not necessarily injective, there may in principle be multiple functions $F$ satisfying $S_F=S^{p_0}$. However, we may single out one such `natural' parameter, which we define as
\begin{equation}\label{F-natural}
    F^{\natural} (x_{1:k})= \Phi^{-1}\big(\partial_k S^{p_0}_{k}(x_{1:k}) \big),~~~ k=1,\dots,d.
\end{equation}
By construction, $S_{F^\natural}=S^{p_0}$. Now choose $K_{min}<c_{\text{mon}}$ and $K_{max}>L$. Since all derivatives of $\Phi^{-1}$ are bounded on the sub-interval $[c_{\text{mon}},L]\subset (K_{min},K_{max})$, by definition of $\KR(\alpha,L,c_{\text{mon}})$, and by Lemma \ref{faadibruno}, we thus overall obtain
\begin{equation}\label{C-alpha-bound}
\begin{split}
    \sup_{p_0\in \mathcal M(\alpha,B,c)} \|F^{\natural}\|_{C^\alpha(Q_d,\R^d)}\le M<\infty,
\end{split}
\end{equation}
for some $M=M(\alpha,B,c,\Phi)$.

We now verify the hypotheses necessary to apply Theorem \ref{thm-penalty-gen} in the (triangular) case \textbf{(i)}, with parameter set $\Theta=H^\alpha_\Delta(Q_d,\R^d)$ and penalty functional $\textnormal{pen}(F)=\|F\|_{H^\alpha_\Delta(Q_d,\R^d)}$. By the standard Sobolev embedding as well as the hypothesis $\alpha>d/2+1$, we have that $\Theta \subseteq C^1(Q_d,\R^d)$ -- combined with (\ref{upper}) and (\ref{lower}) below, this verifies the boundedness requirements from Assumption \ref{generic-assumption} \textbf{(ii)}.

\smallskip

Next, we derive a suitable upper bound $\Psi(\lambda_N,R)$ for the $C^1_{\text{diag}}$-entropies featuring in (\ref{entr-integral}). Recall the functional $\tau^2(F,p)$ from (\ref{tau-def}), here given by
\[ \tau^2(F,p)= h^2(S_F^\#\eta,p)+\lambda_N^2\|F\|_{H^\alpha_\Delta(Q_d,\R^d)}^2. \]
Since $F^\natural\in C^\alpha(Q_d,\R^d)\subseteq \Theta$, we may choose $\theta_* = F^\natural$ in Theorem \ref{thm-penalty-gen}, whence $S_{\theta_*}^\#\eta=p_0$. Thus, the sets $\theta_*(\lambda_N,R)$ from (\ref{star-sets}) are given by
\[\Theta_*(\lambda_N,R)=\{F\in H^\alpha_\Delta(Q_d,\R^d):\tau^2(F,p_0)\le R^2 \},~~~R>0,\]
which in turn implies that each component map $F_k$ satisfies the norm bound
\begin{equation}\label{sob-bound}
    \sup_{F\in \Theta_*(\lambda_N,R)} \|F_k\|_{H^\alpha(Q_k)}\le \lambda_N/R,~~~ R>0,~1\le k\le d.
\end{equation}
Using the increment bound (\ref{increment}) as well as the $\|\cdot\|_\infty$-covering bound from Corollary 2 of \cite{NP07} (applicable here since $\alpha>d/2$)\ssw{Use this to discuss the regularity assumption in Sections 2 and 4?} to each component function $F_k$, we then obtain that for some $c,c'>0$ and any $\rho>0$,
\begin{equation}\label{Halpha-entropy}
    \begin{split}
        H(\mathcal S^*(\lambda_N,R),&\|\cdot\|_{C^1_{\text{diag}}},\rho)\le H(\Theta_*(\lambda_N,R),\|\cdot\|_{L^\infty(Q_d,\R^d)}/c,\rho)\\
        &\lesssim \sum_{k=1}^d H(\{g\in H^\alpha(Q_k): \|g\|_{H^\alpha(Q_k)}\le \lambda_N/R \},\|\cdot\|_{L^\infty(Q_k)}/c', \rho)\\
        &\lesssim \sum_{k=1}^d \max\Big\{1, \Big(\frac{\lambda_N}{\rho R}\Big)^{k/\alpha}\Big\} \lesssim 1+ \Big(\frac{\lambda_N}{\rho R}\Big)^{d/\alpha}.
    \end{split}
\end{equation}
Again using that $\alpha>d/2$, we obtain that
\begin{equation*}
\begin{split}
    \int_0^R H^{1/2}(\mathcal S^*(\lambda_N,R),\|\cdot\|_{C^1_{\text{diag}}},\rho)d\rho &\lesssim R+ \Big(\frac{\lambda_N}{R}\Big)^{d/2\alpha} R^{1-d/2\alpha} \\
    &= R+ \lambda_N^{d/2\alpha} R^{1-d/\alpha}.
\end{split}
\end{equation*}
This gives our upper bound, i.e., $\Psi^*(\lambda,R)\coloneqq C(R+\lambda_N^{d/2\alpha} R^{1-d/\alpha})$ for some sufficiently large $C>0$. We may now re-write the condition (\ref{delta-req}) as
\begin{equation*}
    \sqrt N \delta^2 \gtrsim (\delta+\lambda_N^{d/2\alpha} \delta^{1-d/\alpha}),
\end{equation*}
or equivalently (using that $\lambda_N= N^{-\frac{\alpha}{2\alpha+d}}$),
\[ \delta\gtrsim \max \big\{N^{-1/2}, N^{-\frac{\alpha}{2\alpha+d}} \big\} = N^{-\frac{\alpha}{2\alpha+d}}, \]
and we may thus choose $\delta\simeq N^{-\frac{\alpha}{2\alpha+d}}$. Moreover, (\ref{C-alpha-bound}) implies that
\begin{equation*}
    \begin{split}
        \sup_{p_0\in \mathcal M(\alpha,B,c)} \tau^2(F^\natural, p_0)=\sup_{p_0\in \mathcal M(\alpha,B,c)} \lambda_N^2 \|F^\natural\|_{H^\alpha_\Delta(Q_d,\R^d)}^2 &\lesssim \sup_{p_0\in \mathcal M(\alpha,B,c)} \lambda_N^2 \|F^\natural\|_{C^\alpha(Q_d,\R^d)}^2 \\
        & \lesssim N^{-\frac{2\alpha}{2\alpha+d}}.
    \end{split}
\end{equation*}
whence Theorem \ref{thm-penalty-gen}, applied with $\lambda= N^{-\frac{\alpha}{2\alpha+d}}$ and $\theta_*=F^\natural$, implies the desired concentration inequality.
\end{proof}

The following lemma contains some boundedness and continuity properties of the rational parameterization (\ref{S-F}).

\begin{lem}\label{C1diag-lip}
    Suppose that $0<K_{min}<K_{max}<\infty$, and that $\Phi: \R \to (K_{min},K_{max})$. Then there exists $M>0$ such that for any triangular maps $F,\tilde F\in C^1(Q_d,\R^d)$, we have 
    \begin{align}
        \|S_F\|_{C^1_{\text{diag}}}&\le M, \label{upper}\\
        \inf_{x\in Q_d} \Pi_{k=1}^d \partial_kS_{F,k}(x_{1:k}) &\ge M^{-1}, \label{lower}
    \end{align}
    as well as the increment bound
    \begin{equation}\label{increment}
        \|S_{F}-S_{\tilde F}\|_{C^1_{\text{diag}}}\le M \|F-\tilde F\|_{L^\infty(Q_d,\R^d)}.
    \end{equation}
\end{lem}

\begin{proof}
    For any $1\le k\le d$, the definition (\ref{S-F}) of $S_{F,k}:Q_k\to \R$ implies that
    \begin{equation}\label{lip-const}
        \begin{split}
            \max\{ \|S_{F,k}\|_{\infty},\|\partial_kS_{F,k}\|_\infty \}&\lesssim \sup_{x\in Q_k}\Big|\frac{\Phi \circ F_k (x) dy}{\int_{0}^{1}\Phi \circ F_k (x_{k-1}, y) dy}\Big|\\
            &\lesssim \|\Phi\circ F\|_\infty \|(\Phi\circ F)^{-1} \|_\infty\\
            &\le K_{max}K_{min}^{-1},
        \end{split}
    \end{equation}
    which proves (\ref{upper}).
    
    Next, (\ref{lower}) follows from noting that $\partial_kS_k\ge K_{min}/(2K_{max})$. Finally, for (\ref{increment}), observe that for any $F,\tilde F$ and $x\in Q_k$,
    \begin{equation*}
    \begin{split}
        \big|\partial_kS_{F,k}(x)-\partial_kS_{\tilde F,k}(x)\big|&=\Big|\frac{\Phi \circ F_{k} (x) dy}{\int_{0}^{1}\Phi \circ F_{k} (x_{k-1}, y) dy}- \frac{\Phi \circ \tilde F_{k} (x) dy}{\int_{0}^{1}\Phi \circ \tilde F_{k} (x_{k-1}, y) dy}\Big|\\
        &\lesssim \|F_{k}-\tilde F_{k}\|_\infty.
    \end{split}
    \end{equation*}
    Integrating this in the $x_k$-variable yields a similar estimate for $\|S_{F,k}-S_{\tilde F,k}\|_\infty$, completing the proof.
\end{proof}

We now turn to the proof of the Theorem \ref{thm-wave} regarding wavelet-based estimators.

\begin{proof}[Proof of Theorem \ref{thm-wave}]
1. \textit{`Bias term' in Hellinger distance}. We begin by constructing an appropriate element $\theta_*\in\Theta$, with the goal to bound $\tau^2(\theta_*,p_0)$ given by (\ref{tau-def}). For any $s\ge 0$, $1\le p,q\le \infty$, we denote the usual Besov spaces on $\R^k$ and $Q_k$ by $B^s_{pq}(\R^k)$ and $B^s_{pq}(Q_k)$ respectively; see for instance \cite{T08} for definitions. For $p_0\in \mathcal M(\alpha,B,c)$, let $F^\natural$ again denote the `natural' parameter from (\ref{F-natural}). We approximate each of the component functions $F^\natural_k: Q_k\to \R$ by wavelet projections.
Let $\mathcal E$ denote a bounded linear extension operator \[\mathcal E: B^\alpha_{\infty\infty}(Q_k)\to B^\alpha_{\infty\infty}(\R^k),~~~~ g\mapsto \mathcal E(g).\] For the existence of such operators on Lipschitz domains, see, e.g., \cite{T06}.

Let us denote by $P_J^k$ the $J$-th resolution level wavelet projection operator on $B^\alpha_{\infty\infty}(\R^k)$,
\[ P_J^k(g)=\sum_{l=0}^J\sum_{m\in \Z^k} \psi_{lm}^k \langle\psi_{lm}^k,g \rangle, ~~~g \in B^\alpha_{\infty\infty}(\R^k),\]
Now let $g\in B^\alpha_{\infty\infty}(\R^k)$ and define $\tilde g\coloneqq P_J(\mathcal E(F))|_{Q_d}$. Then, by the standard characterisation of Besov norms in terms of wavelet decay (see, e.g., p. 370 in \cite{GN16}) and since $2^{J}\simeq  N^{\frac{1}{2\alpha+d}}$, we have that 

\begin{equation*}
\begin{split}
    \|g-\tilde g\|_{L^\infty(Q_k)} &\le \|\mathcal E(g)-P_J^k\mathcal E(g)\|_{L^\infty(\R^k)}\\
    &\lesssim \sup_{l\ge J+1} 2^{lk/2}\sup_{m\in \Z^k} |\langle \mathcal E(g), \psi_{lmx }^k\rangle| \\
    &\le 2^{-J\alpha} \|\mathcal E(g)\|_{B^\alpha_{\infty\infty}(\R^k)}\lesssim N^{-\frac{\alpha}{2\alpha+d}}\|g\|_{B^\alpha_{\infty\infty}(Q_k)}.
\end{split}    
\end{equation*}
Since $F^\natural_k \in C^\alpha(Q_k)$ by (\ref{C-alpha-bound}) and since $C^\alpha(Q_k)\subseteq B^\alpha_{\infty\infty}(Q_k)$ (continuous embedding), we obtain that the above approximation holds for the choice $g=F^\natural_k$ uniformly over all $p_0\in \mathcal M(\alpha,L,c_{\text{mon}})$. Moreover, note that for any $g\in C^\alpha(Q_k)$, denoting $g_{lm}=\langle \psi_{lm}^k, \mathcal E (g)\rangle_{L^2(\R^k)}$, $(l,m)\in I_{k,J}$,
\begin{equation*}
\begin{split}
    \sum_{0\le l\le J}\sum_{1\le m\le L_l^k} 2^{2\alpha l} g_{lm}^2 &\simeq \big\|\sum_{0\le l\le J}\sum_{1\le m\le L_l^k} g_{lm}\psi_{lm}^k \big\|_{B^\alpha_{22}(\R^k)}^2\\
    &\le \|\mathcal E(g) \|_{B^\alpha_{22}(\R^k)}^2\lesssim \|\mathcal E(g) \|_{B^\alpha_{\infty\infty}(\R^k)}^2 \lesssim \|g \|_{B^\alpha_{\infty\infty}(Q_k)}^2.
\end{split}
\end{equation*}

Now we choose $\theta_*\in \Theta$ to be given by the component vectors
\begin{equation}\label{sea-star}
    \theta_k^*=(\langle \psi_{lm}^k,\mathcal E(F^\natural_k) \rangle_{L^2(\R^k)}: 0\le l\le J,~1\le m\le L^k_l), ~~~1\le k\le d.
\end{equation} Combining the above with the fact that $S_{\theta_*}^\#\eta$ and $p_0$ are (uniformly) lower bounded, the local Lipschitz estimate (\ref{pushfwd-cont}) and the bound (\ref{upper}), we obtain that
\begin{align*}
    \tau^2(\theta_*,p_0)&= h^2(S_{\theta_*}^\#\eta,p_0) + N^{-\frac{2\alpha}{2\alpha+d}} \|\theta_*\|_{b^\alpha_{22}}^2\\
    &\le \|(S^{p_0})^\# \eta - S_{\theta_*}^\#\eta \|_\infty^2 + N^{-\frac{2\alpha}{2\alpha+d}}\sum_{k=1}^d \|F_k^\natural\|_{B^\alpha_{\infty\infty}(Q_k)}^2\\
    &\lesssim \| S^{p_0}- S_{\theta_*} \|_{C^1_{\text{diag}}}^2 + N^{-\frac{2\alpha}{2\alpha+d}}.
\end{align*}
For $F_{\theta_*}$ given by (\ref{F-theta}), the increment bound (\ref{increment}) and the regularity bound (\ref{C-alpha-bound}) for $F^\natural$ further imply that
\begin{align*}
\| S^{p_0}- S_{\theta_*} \|_{C^1_{\text{diag}}} \lesssim \|F^\natural - F_{\theta_*} \|_{L^\infty(Q_d,\R^d)} \lesssim N^{-\frac{\alpha}{2\alpha+d}}\|F^\natural\|_{B^\alpha_{\infty\infty}}&\lesssim N^{-\frac{\alpha}{2\alpha+d}}\|F^\natural\|_{C^\alpha(Q_d,\R^d)}\\
&\lesssim N^{-\frac{\alpha}{2\alpha+d}},
\end{align*}
which proves that $\tau^2(\theta_*,p_0)\lesssim N^{-\frac{2\alpha}{2\alpha+d}}$, uniformly over $p_0\in \mathcal M(\alpha,L,c_{\text{mon}})$.
\smallskip

2. \textit{Metric entropy}. As in the proof of Theorem \ref{thm-sobolev}, we need to derive an upper bound for the metric entropy for the sets $\mathcal S_*(\lambda_N,R)$ defined in (\ref{theta-*}). For each $1\le k\le d$, we observe that uniformly over $\theta\in \Theta_*(\lambda_N,R)=\{ \theta\in\Theta: \tau^2(\theta,S_{\theta_*}^\#\eta )\le R^2\}$, we have that
\[ \|F_{\theta,k}\|_{H^\alpha(Q_k)}\lesssim \|F_{\theta}\|_{H^\alpha_\Delta(Q_d,\R^d)}\lesssim \sum_{k=1}^d\|F_{\theta,k}\|_{B^\alpha_{22}(Q_k)}\le \|\theta\|_{b^\alpha_{22}}\le \lambda_N/R. \]
Thus, for the remainder of the proof, we may argue exactly as after (\ref{sob-bound}) in the proof of Theorem \ref{thm-sobolev}; except that at the end we utilize Theorem \ref{thm-penalty-gen} with $\theta_*$ chosen as in (\ref{sea-star}) in place of $F^\natural$.
\end{proof}

\begin{proof}[Proof of Lemma \ref{lem-stability}]
    Let us write $S=S_{p,\eta}$, $\bar S=S_{\bar p, \eta}$. Let $x\in Q_d$ and $k\in \{1,\dots, d\}$. We recall the notations $p_k,\bar p_k$ for the marginal conditionals of $p,\bar p$ from (\ref{nu-k-def}), as well as the corresponding CDFs $F^p_{k,x_{1:k-1}},F^{\bar p}_{k,x_{1:k-1}}$ given by (\ref{Fdef}). Since $\eta$ factorizes, we may write $F^\eta_k=F^\eta_{k,x_{1:k-1}}$ for its corresponding CDF since it is independent of $x_{1:k-1}$. Because $\eta_k$ is lower bounded and $C^1$, $(F^\eta_k)^{-1}$ is uniformly bounded in $C^2$. We then obtain that
\begin{equation}\label{ugly-estimate}
    \begin{split}
    |S_k(x_{1:k})&-\bar S_k(x_{1:k})|=\big|\big[\big(F^\eta_{k}\big)^{-1} \circ F^p_{k,x_{1:k-1}}\big](x_k) -\big[\big(F^\eta_{k}\big)^{-1} \circ F^{\bar p}_{k,x_{1:k-1}}\big](x_k)\big|\\
    &\le \|\big(F^\eta_{k}\big)^{-1}\|_{Lip} \big|F^p_{k,x_{1:k-1}}(x_k)-F^{\bar p}_{k,x_{1:k-1}}(x_k)\big|\\
    &\lesssim \int_{0}^{x_k}\Big| \frac{\tilde p_k(x_{1:k-1},y)}{\tilde p_{k-1}(x_{1:k-1})}- \frac{\tilde{\bar p}_k(x_{1:k-1},y)}{\tilde{\bar p}_{k-1}(x_{1:k-1})}\Big|dy\\
    &\lesssim \big|\tilde p_{k-1}(x_{1:k-1})-\tilde{\bar p}_{k-1}(x_{1:k-1})\big| + \big\|\tilde p_k(x_{1:k-1},\cdot)-\tilde{\bar p}_k(x_{1:k-1},\cdot)\big\|_{L^1([0,1])}\\
    &\lesssim \big\| p(x_{1:k-1},\cdot)-\bar p(x_{1:k-1},\cdot)\big\|_{L^1(Q_{d-k+1})}.
    \end{split}
\end{equation}
    Integrating (the square of) this inequality over $x_{1:k}$, using Minkowski's inequality, concavity of $x\mapsto \sqrt x$ as well as Jensen's inequality, we obtain that
    \begin{align*}
        \|S_k-\bar S_k\|_{L^2(Q_k)}&=\Big( \int_{Q_k} \big\| p(x_{1:k-1},\cdot)-\bar p(x_{1:k-1},\cdot)\big\|_{L^1(Q_{d-k+1})}^2dx_{1:k}\Big)^{1/2}\\
        &\lesssim \|p-\bar p \|_{L^2}.
    \end{align*}
    
    The bound for the `diagonal' derivatives $\|\partial_kS_k-\partial_k\bar S_k\|_{L^2}$ is derived as follows. Using the chain rule, we see that

   \begin{equation*}
        \begin{split}
            |\partial_kS_k(x_{1:k})&-\partial_k\bar S_k(x_{1:k})|\\
            &~~~~~~=\big|\big[\big(F^\eta_{k}\big)^{-1} \circ F_{k,x_{1:k-1}}^p\big]'(x_k)-\big[\big(F^\eta_{k}\big)^{-1} \circ F_{k,x_{1:k-1}}^{\bar p}\big]'(x_k)\big|\\
            &~~~~~~\le \Big(\big[\big(F^\eta_{k}\big)^{-1}\big]' \circ F_{k,x_{1:k-1}}^p(x_k)-\big[\big(F^\eta_{k}\big)^{-1}\big]' \circ F_{k,x_{1:k-1}}^{\bar p}(x_k)\Big){F_{k,x_{1:k-1}}^p}'(x_k)
            \\&~~~~~~~~~~~~~~ +\big[\big(F^\eta_{k}\big)^{-1}\big]' \circ F_{k,x_{1:k-1}}^{\bar p}(x_k) \Big({F_{k,x_{1:k-1}}^p}'(x_k) -{F_{k,x_{1:k-1}}^{\bar p}}'(x_k)\Big)\\
            &~~~~~~\le  \big\|\big(F^\eta_{k}\big)^{-1}\big\|_{C^2}\|{F_{k,x_{1:k-1}}^p}'\|_\infty |F_{k,x_{1:k-1}}^p(x_k) - F_{k,x_{1:k-1}}^{\bar p}(x_k)| \\
            &~~~~~~~~~~~~~~ + \big\|\big(F^\eta_{k}\big)^{-1}\big\|_{C^1} |{F_{k,x_{1:k-1}}^p}'(x_k) - {F_{k,x_{1:k-1}}^{\bar p}}'(x_k)|\\
            &~~~~~~ \lesssim |F_{k,x_{1:k-1}}^p(x_k) - F_{k,x_{1:k-1}}^{\bar p}(x_k)| + |p_k(x_{1:k})- \bar p_k(x_{1:k})|.
        \end{split}
    \end{equation*}
    First term can then be further estimated as in (\ref{ugly-estimate}), and the second term can be bounded by
    \[ |p_k(x_{1:k})- \bar p_k(x_{1:k})|\lesssim \big|\tilde p_{k-1}(x_{1:k-1})-\tilde{\bar p}_{k-1}(x_{1:k-1})\big| + \big|\tilde p_k(x_{1:k})-\tilde{\bar p}_k(x_{1:k})\big|.  \]
    The rest of the proof similarly follows by integrating (the square of) the preceding inequalities, and is thus omitted here.
\end{proof}

\section{Proofs for Section \ref{MLE-res}}\label{MLE-proofs}

We begin with the proofs of Propositions \ref{prop-fact} and \ref{prop-sharp}.

\begin{proof}[Proof of Proposition \ref{prop-fact}]
    Fix $\nu\in \tilde{\mathcal M}(\alpha,B,c)$ (target) and $\mu\in \tilde{\mathcal M}(\alpha,B,c)$ (reference). For ease of notation, we shall write $S= S_{\nu,\mu}$, and we denote the one-dimensional marginal distributions of $\mu$ by $\mu(x)= \Pi_{k=1}^d m_k(x_k)$. Then, from the definition of the $k$-th `marginal conditional' density $\mu_k(x_{1:k})$ given in (\ref{eta-def})--(\ref{nu-k-def}), it is clear that 
    \[ \mu_k(x_{1:k})=\frac{\Pi_{l=1}^k m_l(x_l)}{\Pi_{l=1}^{k-1} m_l(x_l)}= m_k(x_k).\]
    In particular, $\mu_k$ is independent of $x_{1:k-1}$. We may thus write $F^\mu_{k}\equiv F^\mu_{k,x_{1:k-1}}= \int_{-1}^\cdot \mu_k(y)dy$, whence $S_k$ admits the form    
    \begin{equation}\label{simple-S}
        S_k(x_{1:k})= \big[(F^\mu_k)^{-1}\circ F_{k:x_{1:k-1}}^\nu \big](x_k).
    \end{equation}
    
    To bound the anisotropic smoothess of $S_k$, we notice that since the class of densities $\tilde {\mathcal M}(\alpha,B,c)$ is lower bounded, there are constants $C_1,c_1 >0$ such that for any $\mu\in \tilde{\mathcal M(\alpha,B,c)}$, 
    \begin{align}
    \|F^\mu_k\|_{C^{\alpha+1}([0,1])} &\le C_1~~~ \text{and} \notag\\
    \inf_{y\in [0,1]}(F^\mu_k)'(y)=\inf_{y\in [0,1]} m_k(y) &\ge c_1. \label{mk-lb}
    \end{align}
    We claim that the inverse $(F^\mu_k)^{-1}$ of $F^\mu_k$ also belongs to $C^{\alpha+1}$. To begin, (\ref{mk-lb}) implies that $(F^\mu_k)^{-1}:[0,1]\to [0,1]$ is uniformly continuous. Thus, using that
    \begin{equation}\label{inv-rule}
   {(F^\mu_k)^{-1}}'(y)=\frac{1}{{[(F^\mu_k)}' \circ (F^\mu_k)^{-1}] (y)}= \frac{1}{[m_k \circ (F^\mu_k)^{-1}] (y)},~~~ y\in [0,1],
    \end{equation}
    along with Lemma \ref{faadibruno}, the lower bound (\ref{mk-lb}), and the fact that $x\mapsto x^{-1}$ is uniformly Lipschitz on $[c_1,\infty)$, it follows that $(F^\mu_k)^{-1}\in C^1$. Thus, we may again use (\ref{inv-rule}) to infer that $(F^\mu_k)^{-1}\in C^2$, and bootstrapping this argument yields that for some $C_1'>0$,
	\[ \sup_{\mu\in \mathcal M(\alpha,B,c)}\|(F^\mu_k)^{-1}\|_{C^{\alpha+1}} \le C_1'.\]
	
	Combining the preceding claim with (\ref{simple-S}) and Lemma \ref{faadibruno} yields that $S_k\in C^\alpha(Q_k,\R)$. In addition, recalling the notation $\nu_k$ from (\ref{nu-k-def}), we see that
	\begin{equation}\label{dk-Sk}
	   \begin{split}
	   \partial_k S_k(x_{1:k}) &=\big[{(F^\mu_k)^{-1}}'\circ F^\nu_{k,x_{1:k-1}}\big](x_k) {F^\nu_{k,x_{1:k-1}}}'(x_k)\\
	   &=\frac{\nu_k(x_{1:k})}{\big[m_k\circ(F^\mu_k)^{-1}\circ F^\nu_{k,x_{1:k-1}}\big](x_k)}.
	   \end{split}
	\end{equation}
	Since $\nu\ge c$ is lower bounded, we obtain immediately from (\ref{eta-def}) that for any $1\le k\le d$,
	\[ \inf_{x_{1:k}\in Q_k} \tilde \nu_k \ge c ~~~\text{and} ~~~ \tilde \nu_k\in C^\alpha(Q_k),\]
	whence $\nu_k \in C^\alpha(Q_k)$. Thus, applying Lemma \ref{faadibruno} again to (\ref{dk-Sk}) shows that $\|\partial_k S_k\|_{C^\alpha}\le M$ for some large enough $M$. Since $\nu_k$ and $m_k$ are lower and upper bounded, clearly $\inf_{x_{1:k}\in Q_k} \partial_kS_k > c_{\text{mon}}$ for some $c_{\text{mon}}>0$, and the proof is complete.
\end{proof}

\begin{proof}[Proof of Proposition \ref{prop-sharp}]
    First, by definition of $S^\#\eta$ we observe that 
    \[ \inf_{S\in \KR(\alpha, L, c_{\text{mon}}),~x\in Q_d}S^\#\eta(x) \ge c' >0,\]
    for some $c'>0$.
    This follows immediately from the chain rule \eqref{faadibruno} and the standard multiplication inequality for H\"older-type spaces
    \[ \|fg\|_{C^\alpha}\le C \|f\|_{C^\alpha}\|g\|_{C^\alpha}, ~~\text{for some}~ C>0,~\text{all}~f,g,\in C^\alpha.\]
    Indeed, we then have for some increasing function $F$ that
    \begin{equation}
        \begin{split}
        \|S^\# \eta \|_{C^\alpha}&= \| \eta \circ S ~ \det \nabla S\|_{C^\alpha}= \| \eta \circ S\|_{C^\alpha}\big\|\Pi_{k=1}^d\partial_kS_k \big\|_{C^\alpha}\\
        &\lesssim F(\|\eta\|_{C^\alpha},\|S\|_{C^\alpha}) \Pi_{k=1}^d\|\partial_kS_k\|_{C^\alpha},
        \end{split}
    \end{equation}
    which is uniformly bounded over $S\in \KR(\alpha,L,c_{\text{mon}})$.
\end{proof}

Given the previous two propositions, Theorem \ref{thm-anisotropic} will be obtained by utilizing a general convergence result for maximum likelihood estimators developed in \cite{VDG93,VDG00}, which we briefly recall here for the convenience of the reader. Recalling the $L^2$-bracketing metric entropy $H_B(\mathcal F,P_0,\rho)$ for any class $\mathcal F$ of functions $Q_d\to\R$, defined in (\ref{bracketing}), we define the following bracketing metric entropy integral:
\[ J_{B}(\mathcal F, P_0,\delta) \coloneqq \int_{0}^\delta H_B^{1/2}(\mathcal F, P_0,\rho) d\rho. \]

The next proposition, which follows from Theorem 10.13 from \cite{VDG00}, relates the convergence rates of maximum likelihood estimators over any class of densities $\mathcal P$, to the entropy integral of its square root densities,
\[ \mathcal A_{\mathcal P}\coloneqq\Big\{ \sqrt{\frac{p+p_0}{2}}: p\in \mathcal P, h(p,p_0)\le \delta \Big\}, ~~~ \delta>0. \]

\begin{prop}\label{sara}
	Let $\mathcal P$ be a class of densities such that $p_0\in\mathcal P$. Suppose that $\Psi(\delta)\ge J_{B}(\mathcal A_{\mathcal P}, P_0,\delta)$ is an upper bound such that $\Psi(\delta) / \delta^2$ is non-increasing in $\delta$. There exists a constant $C>0$ such that for any $N\ge 1$ and $\delta>0$ satisfying
	\begin{equation}\label{entropy-requirement}
	    \delta^2 \sqrt N \ge C \Psi(\delta)
	\end{equation}
	and all $R\ge \delta$, we have
	\begin{equation}\label{hellinger-focus}
		P_0^n\big( h(\hat p, p_0) \ge R \big) \le C\exp\Big( -\frac{nR^2}{C} \Big).
	\end{equation}
\end{prop}

\begin{proof}
    Inspection of the proof of Theorem 10.13 in \cite{VDG00} shows that we can assert not just  $h(\hat p_n,p_0)=O_{P^n_0}(\delta + h(p_*,p_0))$ (where we have adopted notation from the reference), but also the stronger concentration bound (\ref{hellinger-focus}); see also the proof of Theorem 7.4 in \cite{VDG00}.
\end{proof}

\begin{proof}[Proof of Theorem \ref{thm-anisotropic}]
    Suppose that $p_0\in \mathcal M(\alpha,B,c)$ and $\eta \in 
    \tilde{\mathcal M}(\alpha,B,c)$. Then, by Proposition \ref{prop-fact} there exists $S_0\in \KR (\alpha, L, c_{\text{mon}})$ (for some $L,c_{\text{mon}}>0$) such that $S_0^\#\eta=p_0$. We fix such $L,c_{\text{mon}}$, noting that these constants may be chosen uniformly over $p_0\in \mathcal M(\alpha,B,c)$. Moreover, by Proposition \ref{prop-sharp}, we have the set inclusion
    \[ \big\{S^\#\eta: S\in \KR(\alpha, L, c_{\text{mon}})\big\}\subseteq \mathcal M( \alpha,\tilde B,\tilde c), \]
    for some $\tilde B,\tilde c>0$. Since $p_0$ is bounded away from zero and since $x\mapsto \sqrt x$ is globally Lipschitz on $[l, \infty)$ for any $l>0$, the class of square root densities
    \[\mathcal A= :\Big\{\sqrt \frac{S^\#\eta +p_0}2:S\in S\in \KR (\alpha, L, c_{\text{mon}}) \Big\}\]
    is embedded in some ball of $C^{\alpha}(Q_d, \R)$. Thus, using Corollary 3 in \cite{NP07} (and noting the remark in Section 3.3.4 of \cite{NP07}), we respectively obtain the bracketing metric entropy bound and entropy integral estimate
    \[H_B(\rho,\mathcal A,p_0)\lesssim \rho^{d/\alpha},~~~\int_{0}^\delta H_B^{1/2}(\rho,\mathcal A,p_0)d\rho \lesssim \delta^{\frac{2\alpha-d}{2\alpha}}.\]
    Standard calculations then yield that the requirement (\ref{entropy-requirement}) translates to the condition
    \[ \delta \gtrsim N^{-\frac{\alpha}{2\alpha+d}}. \]
    Since $S_0\in \KR (\alpha,L,c_{\text{mon}})$, we may apply Proposition \ref{sara} with $p_*=p_0$ and $\delta \equiv \delta_N \simeq N^{-\frac{\alpha}{2\alpha+d}}$, which completes the proof of Theorem \ref{thm-anisotropic}.
\end{proof}

The proof of the following basic lemma follows from a generalized Fa\`a di Bruno's formula, such as stated in \cite{EM03}, Theorem 1.

\begin{lem}\label{faadibruno}
	Suppose $U\subseteq \R^{n},~ V\subseteq \R^{m}$ are two open subsets. Then, for each $B<\infty$ there exists a constant $C$ such that for any $r$-times ($r\in \N$) continuously differentiable functions $U\xrightarrow{F}V\xrightarrow{G}\R$ with $\max\{\|F\|_{C^\alpha}, \|G\|_{C^\alpha}\}\le B$, we have that $ \|G\circ F\|_{C^\alpha}\le C$.
\end{lem}

\section*{Acknowledgements}

SW and YM acknowledge support from the Air Force Office of Scientific Research Multidisciplinary University Research Initiative (MURI) project ANSRE. We also thank Ricardo Baptista, Thibaut Horel, Bamdad Hosseini, and Richard Nickl for helpful discussions.

\bibliography{tvbib,atmreferences}
\bibliographystyle{abbrv}

\end{document}